\documentclass[11pt,a4paper]{article}

\usepackage[utf8]{inputenc}         
\usepackage{amssymb,amsmath,amsthm} 
\usepackage[UKenglish]{isodate}     
\usepackage[scale=0.8]{geometry}    
\usepackage[scr]{rsfso}
\usepackage{graphicx}               
\usepackage{xcolor}                 
\usepackage{enumerate, enumitem}    

\usepackage{hyperref}               
\hypersetup{colorlinks=true,citecolor=green,urlcolor=blue,linkcolor=blue,pdfauthor={Bandyopadhyay, Lo, Mycroft},pdftitle = {Towards Posa's Conjecture for 3-graphs}}

\usepackage[backend=biber,natbib=true,style=numeric,isbn=false,minbibnames=99,maxbibnames=99]{biblatex} 
\DeclareFieldFormat{title}{\mkbibquote{#1}}
\DeclareFieldFormat[misc]{date}{Preprint (#1)}

\renewbibmacro*{doi+eprint+url}{
	\iftoggle{bbx:url}
	{\iffieldundef{doi}{\usebibmacro{url+urldate}}{}}{}
	\newunit\newblock
	\iftoggle{bbx:eprint}
	{\usebibmacro{eprint}}
	{}
	\newunit\newblock
	\iftoggle{bbx:doi}
	{\printfield{doi}}
	{}
}
\renewbibmacro{in:}{}
\DeclareFieldFormat{pages}{#1}

\newcommand{\s}[1]{\left\lvert #1 \right\rvert}
\newcommand{\J}{\mathcal{J}}
\newcommand{\eps}{\varepsilon}
\newcommand{\N}{\mathbb{N}}
\newcommand{\C}{\mathcal{C}}
\newcommand{\R}{\mathcal{R}}
\newcommand{\T}{\mathcal{T}}
\newcommand{\F}{\mathcal{F}}
\newcommand{\Hy}{\mathcal{H}} 
\newcommand{\cG}{\mathcal{G}} 
\newcommand{\Hskel}{\Hy^{\text{skel}}}
\DeclareMathOperator{\ex}{\mathrm{ex}}

\newtheorem{theorem}{Theorem}[section]
\newtheorem{lemma}[theorem]{Lemma}
\newtheorem{conjecture}[theorem]{Conjecture}
\newtheorem{definition}[theorem]{Definition}

\newtheorem{claim}[theorem]{Claim}
\newtheorem{proposition}[theorem]{Proposition}
\newtheorem{fact}[theorem]{Fact}

\newenvironment{proofclaim}[1][Proof of claim]{\begin{proof}[#1]}{\end{proof}}

\numberwithin{equation}{section}

\setlist[enumerate]{noitemsep,nolistsep}

\newcommand{\EMAIL}[1]{\textit{{E-mail}}: \texttt{\href{mailto:#1}{#1}}} 

\title{Towards P\'osa's Conjecture for~$3$-graphs}

\author{
Debmalya Bandyopadhyay 
\thanks{School of Mathematics, University of Birmingham, \EMAIL{d.bandyopadhyay@bham.ac.uk}} 
\and Allan Lo \thanks{School of Mathematics, University of Birmingham, \EMAIL{s.a.lo@bham.ac.uk}. }
\and Richard Mycroft \thanks{School of Mathematics, University of Birmingham, \EMAIL{r.mycroft@bham.ac.uk}. }
}
 
\date{\today}

\begin{document}
\maketitle

\begin{abstract}
We prove that every $3$-graph $H$ on $n$ vertices with minimum codegree $\delta_2(H) \geq 7n/9 + o(n)$ contains the square of a tight Hamilton cycle. This strengthens a theorem of Bedenknecht and Reiher that $\delta_2(H) \geq 4n/5 + o(n)$ is sufficient. The central novelty of our arguments is an improved understanding of the connectivity structure of~$3$-graphs with large minimum codegree. 
\end{abstract}

\section{Introduction} \label{sec: Introduction}
A fundamental question in graph theory is to determine optimal conditions on a host graph $G$ --- often expressed in terms of the minimum degree $\delta(G)$ --- which guarantee the existence of a given structure in~$G$. Dirac's theorem that every graph $G$ on $n \geq 3$ vertices with $\delta(G) \geq \frac{n}{2}$ admits a Hamilton cycle is an archetypal result of this kind. Since then, a vast body of theory has been developed describing sufficient conditions for embeddings in graphs, hypergraphs, directed graphs and other combinatorial settings. One of the most influential problems in developing this theory was the P\'osa--Seymour conjecture stating that every graph $G$ on $n \geq 3$ vertices with \(\delta(G) \ge \frac{r}{r+1}n\) contains the $r$-th power $\C^r$ of a Hamilton cycle, that is, the graph obtained from a Hamilton cycle by adding the edges of a copy of $K_{r+1}$ on each set of $r+1$ consecutive vertices of the cycle
(P\'osa's conjecture was the case $r=2$ concerning squares of Hamilton cycles, and Seymour~\cite{SeymourConj} later proposed the more general statement).

The difficulty of this conjecture arises principally due to the high level of connectivity of powers of cycles. Indeed, a perfect $K_{r+1}$-tiling (a spanning collection of vertex-disjoint copies of $K_{r+1}$) can be viewed as a less-connected analogue of~$\C^r$, since locally each clique has the same structure as $r+1$ consecutive vertices of~$\C^r$. The Hajnal--Szemer\'edi theorem~\cite{HajnalSzem}, which can be proved by relatively simple techniques, states that the same minimum degree condition as in the P\'osa--Seymour conjecture forces $G$ to contain a perfect $K_{r+1}$-tiling. By contrast, after a great deal of attention and partial results by many researchers~\cite{Fan-Kier2/3,faudree, approxPosaSeymour,PosaConjProof}, it was only by using extensive novel machinery including Szemer\'edi's Regularity lemma and the Blow-up lemma that the P\'osa--Seymour conjecture was finally proven for sufficiently large $n$ by Koml\'os, Sark\"ozy and Szemer\'edi~\cite{PosaSeymourProof}. For small values of~$n$ the conjecture remains open to this day.

The question we consider here is the generalisation of the P\'osa--Seymour conjecture to $3$-graphs (i.e. $3$-uniform hypergraphs --- see Section~\ref{sec:notation} for this definition and for other standard terms pertaining to hypergraphs). Specifically, we consider minimum codegree conditions which ensure the existence of the square of a tight Hamilton cycle in a $3$-graph. The \emph{minimum codegree} of a $3$-graph~$H$, denoted by $\delta_2(H)$, is the largest $m$ for which every pair of vertices of $H$ is contained in at least $m$ edges. A $3$-graph contains the \emph{square of a tight Hamilton cycle} if its vertices can be cyclically ordered so that every set of four consecutive vertices forms a tetrahedron (that is, a copy of $K^3_4$, the complete $3$-graph on four vertices). 
Our main result is that a minimum codegree of~$7n/9 + o(n)$ suffices for this purpose.
\begin{theorem} \label{thm: ourthm}
    For all $\alpha > 0$ there exists $n_0 \in \N$ such that every $3$-graph~$H$ on $n \geq n_0$ vertices with $\delta_2(H) \geq (7/9 + \alpha) n$ contains the square of a tight Hamilton cycle. 
\end{theorem}
This strengthens a previous result of Bedenknecht and Reiher~\cite{ReiherBed}, who showed that a minimum codegree of~$4n/5 + o(n)$ is sufficient. Pavez-Sign\'e, Sanhueza-Matamala and Stein~\cite{hypPosaSeymour} subsequently generalised their argument to give a bound for the $r$-th power of a tight Hamilton cycle in a $k$-graph for each~$r$ and $k$, with the same $4n/5 + o(n)$ bound for the square of a tight Hamilton cycle in a $3$-graph.

As in the graph case, the central difficulty in proving Theorem~\ref{thm: ourthm} arises due to the high level of connectedness of the structure we are aiming to embed. Indeed, the best-possible minimum codegree condition for a perfect tetrahedron tiling in a $3$-graph was identified as essentially $3n/4$ by Lo and Markstr\"om~\cite{LoMark} and by Keevash and Mycroft~\cite{KeevashMycroft} (the latter authors in fact gave the precise optimal bound for sufficiently large $n$). Every four consecutive vertices in the square of a tight Hamilton cycle induce a tetrahedron, so the perfect tetrahedron tiling can be seen as a less-connected analogue of the structure we seek to embed. More widely, our understanding of embeddings of spanning structures in hypergraphs is poor in comparison to the well-developed theory of spanning structures in graphs, and a key reason for this is our poor understanding of how to handle connectivity in hypergraphs. Our proof of Theorem~\ref{thm: ourthm} makes a significant contribution on this specific point: the key ideas which enable us to improve on the work of Bedenknecht and Reiher~\cite{ReiherBed} is a better understanding of the connectivity structure of~$3$-graphs satisfying the given minimum codegree condition, as outlined in Section~\ref{sketch}.
We also recommend the surveys by R\"odl and Ruci\'nski~\cite{Rodl-survey} and by Zhao~\cite{Yi-survey} for further discussion of Hamilton cycles and tilings in uniform hypergraphs.

We do not believe that the bound of Theorem~\ref{thm: ourthm} is optimal. Indeed, we conjecture that a condition of~$3n/4 + o(n)$ is likely to be sufficient for the existence of the square of a tight Hamilton cycle in a graph on $n$ vertices. 

\begin{conjecture}\label{conj}
For all $\alpha > 0$ there exists $n_0 \in \N$ such that every $3$-graph~$H$ on $n \ge n_0$ vertices with $\delta_2(H) \ge (3/4 + \alpha)n$ contains the square of a tight Hamilton cycle.
\end{conjecture}

A slight modification of a construction by Pikhurko~\cite{Pikhurko} shows Conjecture~\ref{conj}, if true, would be best-possible up to the $\alpha n$ error term. This is the following proposition, which we prove in Section~\ref{sec:extremal}.

\begin{proposition}\label{construct}
    For each~$n > 4$ there exists a $3$-graph~$H$ on~$n$ vertices with $\delta_2(H) =  \lfloor 3n/4 \rfloor  - 2$ that does not contain the square of a tight Hamilton cycle.
\end{proposition}

The minimum degree condition of Theorem~\ref{thm: ourthm} is crucial to our proof arguments, and significant new ideas would be needed to reduce the gap in minimum codegree between Theorem~\ref{thm: ourthm} and Proposition~\ref{construct}. We consider this to be a highly significant open problem in extremal graph theory, since its resolution appears to require a much better understanding of how to handle connectivity in uniform hypergraphs under minimum codegree conditions, which is likely to prove useful for many other problems in this area.

\subsection{Key ideas and definitions} \label{sketch}

The central new ideas in our work relate to connectivity properties in $3$-graphs. Specifically, we study the tight connectivity of the \emph{tetrahedral graph} $\T(H)$ of a $3$-graph $H$ of large minimum codegree. This is defined to be the $4$-graph with vertex set $V(H)$ whose edges are all sets of four vertices which induce a copy of~$K^3_4$ in~$H$. 

Let $H$ be a $k$-graph. For edges $e, f \in E(H)$, a \emph{tight walk} from~$e$ to~$f$ in~$H$ is a sequence of (not necessarily disjoint) vertices~$W = (v_1, v_2, \dots, v_{\ell})$ with~$e = \{v_1,\dots, v_k\}$ and $f = \{v_{\ell-k+1}, \dots, v_{\ell}\}$ such that~$\{v_i,\dots, v_{i+k-1}\} \in E(H)$ for each~$i \in [\ell -k+1]$. In other words, $W$ begins with the vertices of~$e$ in some order, concludes with the vertices of~$f$ in some order, and each set of $k$ consecutive vertices in~$W$ forms an edge of~$H$. A set~$\mathscr{E} \subseteq E(H)$ of edges of~$H$ is \emph{tightly connected} if for all~$e, f \in \mathscr{E}$ there exists a tight walk from~$e$ to~$f$. Observe that the relation on $E(H)$ in which $e \sim f$ if there exists a tight walk from~$e$ to~$f$ is an equivalence relation; the equivalence classes of this relation are the \emph{tight components} of~$H$. We say that $H$ is tightly connected if $E(H)$ is \emph{tightly connected}, meaning that $H$ has a single tight component.

It is straightforward to show that if $H$ is a $3$-graph on $n$ vertices with $\delta_2(H) > 4n/5$, then $\T(H)$ is tightly connected. 
Indeed, note that any two copies of~$K^3_4$ in~$H$ sharing three vertices are tightly connected in~$\T(H)$.
Since $H$ is tightly connected (c.f. \cite[Lemma~2.1]{RRS3graph}), it suffices to show that given any two edges $\{x,y,z\}$ and $\{w,y,z\}$ in~$H$, there exists a vertex~$u$ such that both $\{x,y,z,u\}$ and $\{w,y,z,u\}$ form a copy of~$K^3_4$ in~$H$. Since $\delta_2(H) > 4n/5$, we may do this by choosing $u$ to be a common neighbour of the five pairs $xy$, $xz$, $yz$, $wy$ and~$wz$.
A variant of this argument played a key role in Bedenknecht and Reiher's proof~\cite{ReiherBed}.

A key step in the proof of Theorem~\ref{thm: ourthm} is to show that the weaker minimum codegree condition $\delta_2(H) > 7n/9$ is also sufficient to ensure that $\T(H)$ is tightly connected. This is the following lemma, whose proof is the subject of Section~\ref{sec: Tight Connectivity}.

\begin{lemma} \label{lma: tight connectivity at 7n/9}
    Let~$H$ be a~$3$-graph on~$n$ vertices with~$\delta_2(H) > 7n/9$. Then~$\T(H)$ is tightly connected.
\end{lemma}

A \emph{squared-tight-walk} in a $3$-graph $H$ is an ordered sequence $W = (v_1,v_2, \dots, v_{\ell})$ of vertices of~$H$ in which every set of four consecutive vertices forms a copy of~$K^3_4$ in~$H$. So squared-tight-walks in~$H$ are precisely tight walks in the tetrahedral graph~$\T(H)$. We refer to the ordered triples $(v_1,v_2,v_3)$ and $(v_{\ell-2},v_{\ell-1},v_\ell)$ as the \emph{initial triple} and the \emph{final triple} of~$W$ respectively, and we refer to both triples as \emph{ends} of~$W$. For convenience of notation we often write ordered triples as, for example, $\mathbf{v} = (v_1,v_2,v_3)$. We say that $W$ is \emph{from~$\mathbf{u}$ to~$\mathbf{v}$} if $W$ has initial triple $\mathbf{u}$ and final triple $\mathbf{v}$. Of particular importance are \emph{squared-tight-paths} in~$H$, which are squared-tight-walks in~$H$ in which all vertices in the sequence are distinct. 

Since the minimum degree condition of Lemma~\ref{lma: tight connectivity at 7n/9} ensures that every edge of~$H$ extends to a copy of~$K^3_4$ in~$H$, the conclusion of Lemma~\ref{lma: tight connectivity at 7n/9} can be equivalently reformulated by saying that for all $e, f \in E(H)$ there is a squared-tight-walk from~$e$ to~$f$. However, this squared-tight-walk is not necessarily a squared-tight-path in~$H$. By applying Lemma~\ref{lma: tight connectivity at 7n/9} to the reduced $3$-graph of~$H$ obtained from an application of the Regular Slice Lemma of Allen, B\"ottcher, Cooley and Mycroft~\cite{ABCM}, we can show that indeed there must be a squared-tight-path from~$e$ to~$f$ in~$H$ for all $e, f \in E(H)$. In fact we give a stronger result -- the following connecting lemma -- which allows us to find squared-tight-paths linking each of several pairs of edges. Moreover, we may insist that these squared-tight-paths are pairwise vertex-disjoint and avoid a given small set $X$ of `forbidden' vertices. For a set $\mathscr{P} = \{P_1, \dots, P_{\ell}\}$ of squared-tight-paths, we write $V(\mathscr{P})$ for~$\bigcup_{1\le i \le \ell}V(P_i)$. The asymptotic notation $a \ll b$ used here --- which is standard in this field --- is defined formally in Section~\ref{sec:notation}.

\begin{lemma}[Connecting Lemma] \label{lma: simpler connecting lemma}
    Let $1/n \ll \psi \ll \alpha$. Let $H$ be a $3$-graph on $n$ vertices with $\delta_2(H) \ge (7/9 + \alpha)n$. Let $X \subseteq V(H)$ be such that $|X| \le \psi n$. Let $\mathbf{x}_1, \mathbf{y}_1, \dots, \mathbf{x}_s, \mathbf{y}_s$ be vertex-disjoint edges in~$H \setminus X$ with $s \le \psi n$. Then there exists a set $\mathscr{P} = \{P_1, \dots, P_{s}\}$ of vertex-disjoint squared-tight-paths in~$H\setminus X$ such that for each~$i \le s$, $P_i$ is from~$\mathbf{x}_i$ to~$\mathbf{y}_i$, $|V(\mathscr{P})| \le \psi^{1/2} n$ and $|V(\mathscr{P}) \cap X|=\emptyset$.
\end{lemma}

The derivation of Lemma~\ref{lma: simpler connecting lemma} from Lemma~\ref{lma: tight connectivity at 7n/9} is presented in Section~\ref{sec: Connecting Lemma}, and the definitions and results concerning hypergraph regularity used in this argument are presented in Section~\ref{sec: Hypergraph Regularity}. 

The remaining elements of the proof of Theorem~\ref{thm: ourthm} follow an absorbing strategy, an approach which is now fairly commonplace for extremal problems in combinatorics. This combines the aforementioned connecting lemma with the following two lemmas; we note that Lemma~\ref{lma: simpler connecting lemma} plays a part in the proof of each of these results, and that the proof of Lemma~\ref{lma: path cover lemma} also makes use of hypergraph regularity. 

\begin{lemma}[Absorption Lemma] \label{lma: absorption lemma}
Let~$1/n \ll \beta \ll \alpha$. Let~$H$ be a~$3$-graph on~$n$ vertices with~$\delta_2(H) \ge (7/9 + \alpha)n$. Then there exists a squared-tight-path~$P$ in $H$ on at most~$7\sqrt{\beta} n$ vertices such that for all sets~$L \subseteq V(H) \setminus V(P)$ with~$|L| \le  \beta^2 n$, there is a squared-tight-path~$P'$ with~$V(P') = V(P) \cup L$ such that~$P$ and~$P'$ have the same initial and final triples. 
\end{lemma}

\begin{lemma}[Path Cover Lemma] \label{lma: path cover lemma}
Let~$1/n \ll \gamma \ll \alpha$. Let~$H$ be a~$3$-graph on~$n$ vertices with~$\delta_2(H) \ge (7/9+\alpha)n$. Then for all disjoint pairs of ordered triples of vertices~$\mathbf{e}_1$ and~$\mathbf{e}_2$ that are edges of~$H$, there is a squared-tight-path from~$\mathbf{e}_1$ to~$\mathbf{e}_2$ in~$H$ which covers all but at most~$\gamma n$ vertices of~$H$.
\end{lemma}

We give the proof of Lemma~\ref{lma: absorption lemma} in Section~\ref{sec: Absorption} and the proof of Lemma~\ref{lma: path cover lemma} in Section~\ref{sec: Path Cover Lemma}.
Our main result, Theorem~\ref{thm: ourthm}, then follows by combining Lemmas~\ref{lma: simpler connecting lemma},~\ref{lma: absorption lemma} and~\ref{lma: path cover lemma}. 

\begin{proof}[Proof of Theorem~\ref{thm: ourthm}]
Let $1/n \ll \gamma \ll \beta \ll \alpha$.
By Lemma~\ref{lma: absorption lemma}, there is a squared-tight-path~$P$ in~$H$ on at most~$7\sqrt{\beta} n$ vertices with the absorbing property stated in Lemma~\ref{lma: absorption lemma}. Let~$\mathbf{e}_1$ and~$\mathbf{e}_2$ be the initial and final triples of $P$ respectively, so $\mathbf{e}_1$ and $\mathbf{e}_2$ are disjoint edges of $H$. Let $\mathrm{int}(P) = V(P) \setminus \left(V(\mathbf{e}_1) \cup V(\mathbf{e}_2) \right)$. By Lemma~\ref{lma: path cover lemma} with~$H \setminus \mathrm{int}(P)$ playing the role of~$H$, there is a squared-tight-path~$\hat{P}$ from~$\mathbf{e}_2$ to~$\mathbf{e}_1$ in~$H$ which covers all vertices of~$H \setminus \mathrm{int}(P)$ except for a set $L$ of at most~$\gamma n$ vertices. As~$\gamma < \beta^2$, by Lemma~\ref{lma: absorption lemma} there is a squared-tight-path~$P'$ from~$\mathbf{e}_1$ to~$\mathbf{e}_2$ in $H$ such that $V(P') = V(P) \cup L$. Then~$P' \cup \hat{P}$ is the square of a tight Hamilton cycle in~$H$.\end{proof}

\section{Preliminary results and notation} \label{sec:notation}
  We will use constant hierarchies in our statements as follows: the phrase ``$a \ll b$'' means ``for every~$b > 0$ there exists $a_0 > 0$ such that for all $0 < a \leq a_0$ the following statements hold''. Hierarchies with more constants are defined analogously, and we implicitly assume that if a term of the form $1/m$ appears in such a hierarchy then~$m$ is a positive integer. We omit floors and ceilings when they do not affect the calculations. For $n \in \N$ we write $[n] := \{ 1, \dots, n\}$, and for a set $V$ and $k \in \N$ we write $\binom{V}{k}$ to denote the set of all subsets of~$V$ of size~$k$. For notational simplicity, we write $v_1\dots v_k$ for the set~$\{v_1, \dots, v_k\}$ or the ordered $k$-tuple~$(v_1, \dots, v_k)$; it will be clear from the context which interpretation is intended. 
  
A \emph{$k$-uniform hypergraph}~$H$ (or a \emph{$k$-graph} for brevity) consists of a vertex set $V(H)$ and an edge set $E(H) \subseteq \binom{V(H)}{k}$. We make the following definitions for $k$-graphs $H$. For a vertex subset $S \subseteq V(H)$, the \emph{neighbourhood} of~$S$ is $N_H(S) := \{T \subseteq V(H)\setminus S \colon S \cup T \in E(H)\}$ and the \emph{degree} of~$S$ is $\deg_H (S) := |N_H(S)|$. The \emph{minimum $\ell$-degree} of~$H$ is $\delta_{\ell}(H)  = \min \{ \deg_H(S) \colon S \in \binom{V(H)}{\ell}\}$. We also refer to the minimum $(k-1)$-degree $\delta_{k-1}(H)$ as the \emph{minimum codegree} of $H$. Sometimes we wish to consider the minimum degree over just those $(k-1)$-sets which are contained in at least one edge of $H$, for which we define the \emph{minimum positive codegree} of $H$ to be $$\delta_{k-1}^+(H) := \min \left\{\deg_H(S) \colon S \in \binom{V}{k-1}, \deg_H(S) > 0 \right\}.$$ For a set of vertices $U \subseteq V(H)$ we write $H[U]$ to denote the subgraph induced by~$U$, which has vertex set $U$ and edge set $\{e \in E(H) : e \subseteq U\}$. For $k$-graphs~$G$ and~$H$ we write $H \setminus G$ to denote the subgraph obtained by deleting~$V(G)$ from~$H$ and $H-G$ to denote the subgraph obtained by deleting~$E(G)$ from~$E(H)$. 
For a vertex set $U \subseteq V(H)$ we define $H\setminus U := H[V(H)\setminus U]$. 
For vertex sets~$S$ and~$W$, we write $N_H(S, W) = N_H(S) \cap \binom{W}{k- |S|}$ and $\deg_H(S, W)= |N_H(S, W)|$.
For a vertex~$v \in V(H)$, let~$L(v)$ denote the \emph{link graph of~$v$ in~$H$}, which is a~$(k-1)$-graph with vertex set~$V(H) \setminus \{v\}$ and~$E(L(v)) = N_H(v)$.  For an edge $e \in E(H)$ we write $\partial e := \binom{e}{k-1}$, and we write $\partial H$ to denote the $(k-1)$-graph on $V(H)$ with edge set $\bigcup_{e \in E(H)} \partial e$. 

Now let~$H$ be a~$3$-graph on vertex set~$V$, and let $W_1 = x_1 \dots x_{\ell}$ and $W_2 = y_1 \dots y_m$ be tight walks in~$H$. If $x_{\ell-1} = y_1$ and $x_{\ell}=y_2$, then we define $W_1W_2 := x_1\dots x_{\ell}y_3y_4\dots y_{m}$. Observe that $W_1W_2$ is a tight walk in $H$. Furthermore, if $W_1$ and $W_2$ are tight paths in $H$ and $V(W_1) \cap V(W_2) \setminus \{y_1, y_2\} = \emptyset$ then $W_1W_2$ is also a tight path in $H$. Similarly, if $W_1$ and $W_2$ are tight paths in $H$ and we also have $y_{m-1} = x_1$ and $y_{m}=x_2$ and $V(W_1) \cap V(W_2) \setminus \{x_1, x_2, y_1, y_2\} = \emptyset$, then $W_1W_2$ is a tight cycle in $H$.

For~$k, t \in \N$ we write $K_t^k$ for the complete $k$-graph on $t$ vertices, which has vertex set $[t]$ and whose edges are all sets $e \in \binom{[t]}{k}$.

\subsection{Extremal construction} \label{sec:extremal}

The following slight modification of a construction given by Pikhurko~\cite{Pikhurko} demonstrates that, if true, Conjecture~\ref{conj} would be best-possible up to the $\alpha n$ error term.

\begin{proof}[Proof of Proposition~\ref{construct}]
Let $V_1, V_2, V_3$ and $V_4$ be pairwise-disjoint sets of vertices with $\sum_{i \in [4]} |V_i| = n$ and $\lfloor n/4\rfloor \le |V_i| \le  \lceil n/4 \rceil$ for all $i \in [4]$. Let $H$ be the $3$-graph with vertex set $V := \bigcup_{i \in [4]} V_i$ whose edges are all triples $e \in \binom{V}{3}$ which satisfy one of the following mutually exclusive properties.
\begin{enumerate}[label={\rm(\roman*)}]
    \item \label{itm: 1 construction edges} $|e \cap V_1| = 2$, 
    \item \label{itm: 2 construction edges} $|e \cap V_1| = |e \cap V_i| = |e \cap V_j| = 1$ for distinct $i, j \in [4]\setminus\{1\}$,
    \item \label{itm: 3 construction edges} $|e \cap V_i|=3$ for some $i\in [4]\setminus\{1\}$, or
    \item \label{itm: 4 construction edges} $|e \cap V_i| = 1$ and $|e \cap V_j|=2$ for distinct $i, j \in [4]\setminus\{1\}$. 
\end{enumerate}
For $xy \in \binom{V}{2}$ with $x \in V_i$ and $y \in V_j$, we then have 
\begin{align*}
    N (xy) = 
    \begin{cases}
       V \setminus (V_1 \cup \{x,y\}) & \text{if $i =j$,}\\
       V \setminus (V_j \cup \{x\}) & \text{if $i=1 \ne j$,}\\
       V \setminus(V_k \cup \{x,y\}) & \text{if $\{i,j,k\} = [4] \setminus \{1\}$}.
    \end{cases}
\end{align*}
Hence $\delta_2(H) = \lfloor 3n/4 \rfloor  - 2$.

Every copy $K$ of~$K_4^3$ in~$H$ intersects $V_1$ in exactly $0$ or $2$ vertices. Indeed, there are no edges of~$H$ in~$V_1$, so $|K \cap V_1| \le 2$. If $|K \cap V_1| = 1$, then $e := V(K)\setminus V_1$ is an edge of $H$, so our choice of $E(H)$ implies that we cannot have $|e \cap V_i|=1$ for each~$i \in [4]\setminus\{1\}$. We must therefore have $|e \cap V_i| \geq 2$ for some $i \in [4]\setminus\{1\}$, and so $K$ contains an edge of~$H$ with two vertices in~$V_i$ and one in~$V_1$, but again this is not possible by our choice of edges.

Let $\C = v_1v_2\dots v_{\ell}$ be the square of a tight cycle in~$H$. Note that each set of four consecutive vertices of~$\C$ forms a copy of $K_4^3$, and so has either $0$ or $2$ vertices in $V_1$. It follows that either every four consecutive vertices of $\C$ have no vertices in $V_1$, or every four consecutive vertices of $\C$ have precisely two vertices in $V_1$. In the former case we have $\ell \leq |V_2 \cup V_3 \cup V_4| \leq \lceil 3n/4 \rceil$, and in the latter case we have $\ell \le 2|V_1| \le 2\lceil n/4 \rceil$. This proves the proposition.\end{proof}

\subsection{Tur\'an density}

Given a $k$-graph $F$ and $n \in \N$, the \emph{Tur\'an number} $\ex(n, F)$ is the maximum number of edges that an $n$-vertex $k$-graph can have without having~$F$ as a subgraph. The asymptotics of this quantity are more interesting, hence the \emph{Tur\'an density} $\pi(F)$ of~$F$ is defined as $$\pi(F) = \lim_{n \to \infty} \frac{\ex(n, F)}{\binom{n}{k}}.$$ It is not hard to see that this limit exists for every $k$-graph $F$.

Several variants of Tur\'an density have been considered. Mubayi and Zhao~\cite{Mubayi-Zhao} introduced the \emph{codegree Tur\'an number} $\ex_{\mathrm{co}}(n, F)$ as the maximum $d$ such that there exists an $n$-vertex $k$-graph with minimum codegree~$d$ which does not have~$F$ as a subgraph. They defined the \emph{codegree Tur\'an density} of~$F$ as $$\gamma(F) = \lim_{n \to \infty} \frac{\ex_{\mathrm{co}}(n, F)}{n}$$ and proved that this limit exists for every $k$-graph $F$. 

Given a $k$-graph $F$ and a positive integer~$t$, we write $F(t)$ to denote the $t$-blowup of~$F$, which is the $k$-graph formed as follows. For each vertex $v \in V(F)$ let $U_v$ be a set of $t$ vertices, with the sets~$U_v$ being pairwise vertex-disjoint. The vertex set of $F(t)$ is $V := \bigcup_{v \in V(F)} U_v$, and a set $e \in \binom{V}{k}$ is an edge of $F(t)$ if and only if there exists an edge $f \in F$ with $|e \cap U_v| = 1$ for each~$v \in f$.

We use the following result on supersaturation and the Tur\'an density and codegree Tur\'an density of blowups. Supersaturation for Tur\'an density was discovered by Erd\H{o}s and Simonovits~\cite{Supersaturation}, while for codegree Tur\'an density it was shown by Mubayi and Zhao~\cite{Mubayi-Zhao}. The following statement can be deduced by combining results from Keevash's survey~\cite{Keevash} with those of Mubayi and Zhao~\cite{Mubayi-Zhao}.

\begin{theorem}[{cf. \cite[Lemmas 2.1, 2.2]{Keevash}}, {\cite[Proposition 1.4]{Mubayi-Zhao}}] \label{lma: supersaturation, blowup density}~
Let $1/n \ll \beta \ll \alpha, 1/k, 1/m, 1/t$. Let $F$ be a $k$-graph on $m$ vertices and $G$ be a $k$-graph on $n$ vertices with $|E(G)| \ge (\pi(F) + \alpha)\binom{n}{k}$ or $\delta_{k-1}(G) \ge (\gamma(F) + \alpha)n$. Then $G$ contains at least $\beta \binom{n}{t m}$ copies of~$F(t)$.
\end{theorem}

We also use the following result by Balogh, Clemen and Lidický~\cite{Balogh-Clemen-Lidicky} which provides an upper bound on the codegree Turán density of~$K_5^3$.

\begin{theorem}[{\cite[Theorem 1.2, Table 1]{Balogh-Clemen-Lidicky}}] \label{Thm: codegree density of K35}
    $\gamma(K_5^3) \le 0.74$.  
\end{theorem}

\subsection{Probabilistic tools}

Our arguments make use of the following standard probabilistic inequalities. For a positive integer~$n$ and $p \in [0, 1]$, the \emph{binomial distribution} $\mathrm{Bin}(n, p)$ is the number of successes from~$n$ independent trials, where the probability of success in each trial is $p$. Now let $N, n, m$ be positive integers with $n, m \le N$, and fix a set $T$ of size $N$ and a subset $S \subseteq T$ of size $|S| = m$. The \emph{hypergeometric distribution}~${\mathrm{Hyp}}(N, n, m)$ is the distribution of the random variable $X = |S \cap R|$, where $R \subseteq T$ of size $|R| = n$ is selected uniformly at random among all subsets of~$T$ with~$n$ elements.

\begin{lemma}[Markov's inequality cf. {\cite[Equation 1.3]{MR1782847}}] \label{lma: Markov}
    If $t > 0$ and $X$ is a non-negative random variable, then $\mathbb{P}[X \ge t] \le {\mathbb{E}X}/{t}$.
\end{lemma}

\begin{lemma}[Chernoff's inequality cf. {\cite[Remark 2.5]{MR1782847}}] \label{Lma: Chernoff}
    If $0< \eps\le 3/2$ and $X\sim \mathrm{Bin}(n, p)$, then $\mathbb{P}(|X - \mathbb{E}X| \ge \eps \mathbb{E}X) \le 2e^{-\frac{\eps^2}{3}\mathbb{E}X}.$
\end{lemma}

\begin{lemma}[{Hoeffding's inequality cf.~\cite[Theorem 2.10]{MR1782847}}]\label{Lma: Chernoff hyp}
    If $0< \eps \le 3/2$ and $X\sim {\mathrm{Hyp}}(N, n, m)$, then $\mathbb{P}[|X-\mathbb{E}X| \ge \eps \mathbb{E}X] \le 2e^{-\frac{\eps^2}{3}\mathbb{E}X}$.
\end{lemma}

\section{Tight Connectivity} \label{sec: Tight Connectivity}

In this section we prove Lemma~\ref{lma: tight connectivity at 7n/9}, which states that if $H$ is a 3-graph on $n$ vertices with $\delta_2(H) > 7n/9$ then the tetrahedral 4-graph $\T(H)$ of $H$ is tightly connected. 

Let~$\mathscr{T}$ be the set of distinct tight components of~$\T(H)$. Recall that these components $T \in \mathscr{T}$ partition the set of copies of $K^3_4$ in~$H$.
Throughout this section the 3-graphs~$H$ we work with will have the property that every edge of~$H$ is contained in at least one copy of~$K_4^{3}$. 
(Observe that having minimum positive codegree $\delta_2^+(H) > 2n/3$ is a sufficient condition for this.)
Moreover, recalling that the tight components~$T \in \mathscr{T}$ partition the set of copies of~$K^3_4$ in~$H$, we observe that all of the copies of~$K^3_4$ which contain a given edge $e \in E(H)$ must be contained in the same tight component~$T \in \mathscr{T}$. So for each edge $e \in E(H)$ there is precisely one tight component $T \in \mathscr{T}$ which contains the copies of~$K^3_4$ containing~$e$ (of which there is at least one); we denote this tight component by $\phi(e)$, giving an edge-colouring $\phi : E(H) \to \mathscr{T}$ of $H$ (where the colours are the tight components in $\mathscr{T}$).
We identify each tight component $T \in \mathscr{T}$ with the $3$-graph on vertex set $V(H)$ and whose edges are all edges~$e \in E(H)$ with $\phi(e) = T$; in this way we may discuss $N_{T}(S)$, $N_{\partial T } (S)$, $\deg_{T}(S)$ and $\deg_{\partial T} (S)$, as defined earlier. 

For each~$v \in V(H)$, we write $\phi_v(uw) := \phi(uvw)$. 
So $\phi_v$ is an edge-colouring of $L(v)$ (the link graph of~$v$). 
Also, for each~$S \subseteq V(H)$ with $|S| \le 2$, let $\phi(S) = \{\phi(e) \colon  e \in E(H) \text{ and } S \subseteq e\}$, so $\phi(S)$ is the set of colours of edges which contain $S$ as a subset (this set may contain multiple elements). If $T \in \phi(S)$, then we say that $S$ \emph{is in tight component}~$T$. 

Since two copies of $K^3_4$ which share three vertices must lie in the same tight component, we have the following fact.

\begin{fact} \label{fct: common nbr of adjacent triangles}
    Let~$H$ be a~$3$-graph. 
    Suppose that $e, e' \in E(H)$ have $|e\cap e' | = 2$ and $\bigcap_{P \in \partial e \cup \partial e'} N_H(P) \neq \emptyset$.
    Then~$e$ and~$e'$ are in the same tight component of~$\T(H)$. 
\end{fact}

Our definition of $\phi(S)$ ensures that $\phi(uv) \subseteq \phi(u) \cap \phi(v)$ and $\phi(uvw) \in \phi(uv) \cap \phi(uw) \cap \phi(vw)$. The following statement follows immediately, since if $\delta_2(H) > 0$ then for each pair $u,v$ of distinct vertices of $H$ there is an edge of $H$ containing both $u$ and $v$.

\begin{fact} \label{fct: colours are intersecting family}
    Let $H$ be a~$3$-graph on~$n$ vertices with $\delta_2(H) > 0$ in which each edge is contained in a copy of $K^3_4$. 
    Then for all~$u, v \in V(H)$ we have~$\phi(u) \cap \phi(v) \neq \emptyset$. 
In particular, if there exists $u \in V(H)$ satisfying $|\phi(u)| = 1$, then $\phi(u) \subseteq \phi(v)$ for all $v \in V(H)$.     
\end{fact}

The next fact gives a lower bound on the number of vertices covered by each tight component, and holds because for each pair $a,b$ of distinct vertices of an edge~$e$ there are at most $n - \delta^+_2(H)$ vertices $u \in V(H)$ for which $abu$ is not an edge of $H$, and so at least $n - 3(n - \delta^+_2(H))$ vertices of~$H$ form an edge with all three pairs of vertices of $e$. 
Recall that $\delta_2^+(H)$ is the minimum positive codegree of~$H$, and (for the second part of the statement) that all the copies of $K^3_4$ extending a given edge lie in the same $T \in \mathscr{T}$.

\begin{fact}\label{fct: tetrahedral codegree}
    Let~$\alpha>0$ and~$H$ be a~$3$-graph on~$n$ vertices with~$\delta_2^+(H) \ge (7/9+\alpha)n$. Then every edge~$e \in E(H)$ has $\deg_{\T(H)}(e) \geq (1/3+3\alpha)n$. 
    If additionally $\delta_2(H) >0$, then for all $xy \in \binom{V(H)}2$ and $T \in \phi(xy)$ we have $|V(T)| \geq \deg_{T}(xy) \ge (1/3+3\alpha)n$. 
\end{fact}

Our next proposition asserts that if $\T(H)$ has at least two tight components, then some pair of vertices of $H$ is in at least two of these tight components.

\begin{proposition} \label{prp: weird configurations exist}
    Let $H$ be a~$3$-graph on~$n$ vertices with~$\delta_2(H) > 7n/9$ for which $\T(H)$ has at least two tight components. 
    Then there exists $xy \in \binom{V(H)}2$ such that $|\phi(xy)| \ge 2$.
    Moreover, for all $xy \in \binom{V(H)}2$, all $T_1, T_2 \in \phi(xy)$ and all $z \in N_{T_1}(xy)$, there exists $w_z \in V(H) \setminus \{x, y, z\}$ with $w_zxy, w_zyz \in E(H)$ and $\phi(w_zxy) = T_2$. 
\end{proposition}

\begin{proof}
By Fact~\ref{fct: colours are intersecting family}  there must exist $x \in V(H)$ with $|\phi(x)| \ge 2$, as otherwise by choosing vertices $u_1$ and $u_2$ with $\phi(u_1) \ne \phi(u_2)$ we would have $\phi(u_1) \cap \phi(u_2) = \emptyset$. Fix such an $x$ and distinct tight components $T_1, T_2 \in \phi(x)$; we may then choose $v_1, v_2 \in V(H)\setminus \{x\}$ with $T_1 \in \phi(xv_1)$ and $T_2 \in \phi(xv_2)$.
If $|\phi(xv_i)| \ge 2$ for some $i \in [2]$, then we may take $y = v_i$; otherwise, we may choose $y \in N(xv_1) \cap N(xv_2)$. 

For the second part of the statement, let $xy \in \binom{V(H)}2$ have $T_1, T_2 \in \phi(xy)$ and choose $z \in N_{T_1}(xy)$.
By Fact~\ref{fct: tetrahedral codegree} we have $\deg_{T_2}(xy) >n/3$, so we may choose $w_z \in N_H(yz) \cap N_{T_2}(xy)$. 
\end{proof}

Our next proposition shows that in any collection of nine pairs of vertices of $H$ there are some eight pairs which share a common neighbour. Similarly, given a nominated pair of vertices of $H$ and a further seven pairs, we can find seven pairs including the nominated pair which share a common neighbour. For notational simplicity, for  a set $\mathcal{P} \subseteq \binom{V(H)}{2}$ and $P \in \mathcal{P}$, we write $\mathcal{P}\setminus P$ to mean~$\mathcal{P}\setminus\{P\}$.

\begin{proposition} \label{Prp: 8 from 9 and 7 with one fixed}
Let~$H$ be a~$3$-graph on~$n$ vertices with~$\delta_2(H) > 7n/9$. 
\begin{enumerate}[label={\rm(\roman*)}]
\item  Let~$\mathcal{P} \subseteq \binom{V(H)}{2}$ have~$|\mathcal{P}| \le 9$. Then some $P' \in \mathcal{P}$ has $\bigcap_{P \in \mathcal{P}\setminus P'} N_H(P) \neq \emptyset$. \label{itm: i - 8 pairs common nbr}

\item Let~$\mathcal{P} \subseteq \binom{V(H)}{2}$ have~$|\mathcal{P}| \le 8$ and fix~$P \in \mathcal{P}$. Then some $P' \in \mathcal{P}\setminus P$ has $\bigcap_{P'' \in \mathcal{P}\setminus P'} N_H(P'') \neq \emptyset.$ \label{itm (ii): P and 6 other pairs} 
\end{enumerate}
\end{proposition}

\begin{proof}
For \ref{itm: i - 8 pairs common nbr}, note that $\sum_{ P \in \mathcal{P}} \deg_H(P) > |\mathcal{P}| (7n/9) \geq (|\mathcal{P}|-2) n$. Since $H$ has $n$ vertices, by averaging some~$x \in V(H)$ must be contained in $N_H(P)$ for at least $|\mathcal{P}| - 1$ of the pairs $P \in \mathcal{P}$.

    For \ref{itm (ii): P and 6 other pairs}, fix $W \subseteq N_H(P)$ with $|W| = 7n/9$, and observe that 
 $$\sum_{P'' \in \mathcal{P} \setminus P} |N_H(P'') \cap W| > (|\mathcal{P}| - 1)(5n/9) \geq (|\mathcal{P}| - 3)7n/9 =  (|\mathcal{P}| - 3)|W|.$$
By averaging we deduce that some vertex~$x \in W$ is contained in~$N_H(P'') \cap W$ for at least $|\mathcal{P}| - 2$ of the pairs $P'' \in \mathcal{P} \setminus P$.
\end{proof}

Our proof of Lemma~\ref{lma: tight connectivity at 7n/9} proceeds through three key steps. First we show that every vertex of $H$ is in at most two tight components of $\T(H)$ in Section~\ref{subsec: every vtx in 2 tight comp}. We then show that $\T(H)$ has at most two tight components in Section~\ref{subsec: at most 2 tight comp}. Finally, we show that in fact $\T(H)$ has only one tight component in Section~\ref{subsec: one tight comp}. 

\subsection{Every vertex is in at most two tight components} \label{subsec: every vtx in 2 tight comp}

The aim of this section is to prove Lemma~\ref{lma: every vtx sees 2 colours}, which states that every vertex of $H$ is in at most two tight components of $\T(H)$.
We proceed by studying the properties of the edge-coloured link graph of a vertex in the following lemma, leading to the conclusion that each pair of distinct vertices is in at most two tight components.

\begin{lemma} \label{lma: link graph properties}
    Let~$H$ be a~$3$-graph on~$n$ vertices with~$\delta_2(H) > 7n/9$. Then for all~$v \in V(H)$ the following statements hold for the edge colouring $\phi_v$ of $L(v)$.
    \begin{enumerate}[label={\rm(\roman*)}]
    \item There is no properly coloured copy of~$C_4$ in~$L(v)$. \label{itm: i - no properly col C4}
    \item There is no copy of~$C_4$ in~$L(v)$ with~$3$ colours.\label{itm: ii - no C4 with three colours}
    \item There is no copy of~$P_4$ in~$L(v)$ with~$3$ colours. \label{itm: iii - no rainbow P4}
    \item For each $u \in V(H) \setminus \{v\}$ we have $|\phi(uv)|\le 2$. \label{itm: iv - every vtx sees 2 colours} 
    \end{enumerate}
\end{lemma}
\begin{proof}

For~\ref{itm: i - no properly col C4} and~\ref{itm: ii - no C4 with three colours}, observe that if a copy of $C_4$ in $L(v)$ is properly coloured or has three colours, then we may write its vertices as $u_0u_1u_2u_3$ in such a way that $\phi_v(u_{i-1}u_i) \neq \phi_v(u_iu_{i+1})$ for each $i \in [3]$. 
Let~$\mathcal{P} = \{vu_{i} \colon i \in [4]\} \cup \{u_iu_{i+1} \colon i \in [4]\}$, with addition in the index performed modulo~4, so $|\mathcal{P}| = 8$. By Proposition~\ref{Prp: 8 from 9 and 7 with one fixed}\ref{itm (ii): P and 6 other pairs} with~$vu_2$ playing the role of~$P$, we deduce that there exists a pair $P' \in \mathcal{P}\setminus \{vu_2\}$ and a vertex~$x$ such that $x \in \bigcap_{P \in \mathcal{P} \setminus P'} N_H(P)$. It follows that for some $i \in [3]$ we have $$x \in \bigcap_{P \in \partial vu_{i-1} u_{i} \cup \partial vu_{i}u_{i+1}} N_H(P),$$ so $\phi_v(u_{i-1} u_{i}) = \phi_v(u_{i}u_{i+1})$ by Fact~\ref{fct: common nbr of adjacent triangles}, giving a contradiction.

For \ref{itm: iii - no rainbow P4}, suppose for a contradiction that~$wxyz$ is a path on~$4$ vertices in~$L(v)$ with $3$ colours. 
By Fact~\ref{fct: tetrahedral codegree} we have $\deg_{ \phi(vyz)}(vy) > n/3$. So we may choose a vertex $u \in N_{ \phi(vyz)}(vy) \cap N_{H} (vw)$, and then~$uwxy$ forms a copy of~$C_4$ in~$L(v)$ with at least~$3$ colours, contradicting~\ref{itm: ii - no C4 with three colours}. 

For \ref{itm: iv - every vtx sees 2 colours}, suppose for a contradiction that some $v \in V(H) \setminus \{v\}$ has $|\phi(uv)| \geq 3$. Choose distinct tight components $T_1, T_2, T_3 \in \phi(uv)$ and vertices $v_1, v_2$ and $v_3$ with~$uvv_i\in T_i$ for each~$i \in [3]$. Let~$\mathcal{P} = \{uv, uv_i, vv_i \colon i \in [3]\}$, so $|\mathcal{P}| = 7$. By Proposition~\ref{Prp: 8 from 9 and 7 with one fixed}\ref{itm (ii): P and 6 other pairs} with~$uv$ playing the role of~$P$, we obtain $P' \in \mathcal{P} \setminus \{uv\}$ for which $\bigcap_{P \in \mathcal{P} \setminus P'} N_H(P) \neq \emptyset$. It follows that there exist distinct $i, j \in [3]$ such that $\bigcap_{P \in \partial uvv_i \cup \partial uvv_j} N_H(P) \neq \emptyset$. By Fact~\ref{fct: common nbr of adjacent triangles} we then have $T_i = \phi(uvv_i) = \phi(uvv_j) = T_j$, a contradiction. 
\end{proof}

\begin{lemma} \label{lma: every vtx sees 2 colours}
    Let~$H$ be a~$3$-graph on~$n$ vertices with~$\delta_2(H) > 7n/9$. Then every vertex is in at most two tight components of~$\T(H)$, that is, $|\phi(v)|\le 2$ for all $v \in V(H)$. 
\end{lemma}

\begin{proof}
    Suppose for a contradiction that a vertex~$v \in V(H)$ is in three tight components~$T_1, T_2, T_3$ of~$\T(H)$. For each~$i \in [3]$ we have $\deg_{\partial T_i}(v) > n/3$ by Fact~\ref{fct: tetrahedral codegree}, so we must have
    $N_{\partial T_i}(v) \cap N_{\partial T_j}(v) \neq \emptyset$ for some distinct~$i, j \in [3]$. Assume without loss of generality that $N_{\partial T_1}(v) \cap N_{\partial T_2}(v) \neq \emptyset$ and fix~$u \in N_{\partial T_1}(v) \cap N_{\partial T_2}(v)$ and $w \in N_{H}(uv) \cap N_{\partial T_3}(v)$. 
    
    By Lemma~\ref{lma: link graph properties}\ref{itm: iv - every vtx sees 2 colours} we have~$\phi_v(uw) = \phi(uvw) \in \phi(uv) = \{T_1, T_2\}$; assume without loss of generality that~$\phi_v(uw) = T_1$. 
    Let~$u' \in N_{T_2}(uv)$ and $w' \in N_{T_3}(vw)\setminus \{u'\}$. Then~$u'uww'$ is a copy of~$P_4$ in~$L(v)$ with~$\phi_v(u'u) = T_2$, $\phi_v(uw) = T_1$ and~$\phi_v(ww') = T_3$, a contradiction to Lemma~\ref{lma: link graph properties}\ref{itm: iii - no rainbow P4}.  
\end{proof}

\subsection{There are at most two tight components}\label{subsec: at most 2 tight comp}

The aim of this section is to prove Lemma~\ref{Lma: two tight components}, which states that $\T(H)$ has at most two tight components. For this we will use the following proposition asserting that $H$ cannot contain 3 vertices which each have a different set of two out of three colours. 

\begin{proposition} \label{prop: no alternating triangle}
    Let~$H$ be a~$3$-graph on~$n$ vertices with~$\delta_2(H) >7n/9$ and~$\mathscr{T}^* = \{T_1, T_2, T_3\}$ be a set of three distinct tight components of~$\T(H)$. Then there do not exist three vertices~$v_1, v_2, v_3$ of $H$ such that~$\phi(v_i) = \mathscr{T}^*\setminus \{T_i\}$ for each $i \in [3]$. 
\end{proposition}
\begin{proof}
    Suppose for a contradiction that there exist~$v_1, v_2, v_3 \in V(H)$ such that~$\phi(v_i) = \mathscr{T}^*\setminus \{T_i\}$ for each $i \in [3]$. Choose $u \in N_H(v_1v_2) \cap N_H(v_2v_3) \cap N_H(v_1v_3)$, which exists since $$|N_H(v_1v_2) \cap N_H(v_2v_3) \cap N_H(v_1v_3)| \ge 3\delta_2(H) - 2n > 0.$$ For all distinct~$i,j \in [3]$ we have $\phi(uv_iv_j) \in \phi(v_iv_j) \subseteq \phi(v_i) \cap \phi(v_j) = \mathscr{T}^*\setminus \{T_i, T_j\}$, so $\mathscr{T}^* \subseteq \phi(u)$, contradicting Lemma~\ref{lma: every vtx sees 2 colours}.
\end{proof}

\begin{lemma} \label{Lma: two tight components}
    Let~$H$ be a~$3$-graph on~$n$ vertices with~$\delta_2(H) > 7n/9$. Then~$\T(H)$ has at most two tight components. 
\end{lemma}

\begin{proof}
    Suppose for a contradiction that $\T(H)$ has at least $3$ distinct tight components. 
    By Proposition~\ref{prp: weird configurations exist} and Lemma~\ref{lma: link graph properties}\ref{itm: iv - every vtx sees 2 colours}, there exists $xy \in \binom{V(H)}2$ such that $|\phi(x y)| = 2$. Let $\phi(x y) = \{T_1, T_2\}$, so $\phi(x) = \phi(y) = \{T_1, T_2\}$ by Lemma~\ref{lma: every vtx sees 2 colours}. Let $T_3$ be a tight component in~$\T(H)$ distinct from~$T_1$ or~$T_2$. By Fact~\ref{fct: tetrahedral codegree} we have $|V(T_3)| > n/3$, so we may choose $z \in N_H(xy) \cap V(T_3)$. Since $\phi(xyz) \in \phi(xy) = \{T_1, T_2\}$, we assume without loss of generality that $\phi(xyz)= T_1$, so $\phi(z) = \{T_1, T_3\}$ by Lemma~\ref{lma: every vtx sees 2 colours}. Also by Fact~\ref{fct: tetrahedral codegree} we have $\deg_{\partial T_3}(z) > n/3$, so we may choose $w \in N_{\partial T_3}(z) \cap N_H(yz)$, and then we have $\phi(wyz) \in \phi(yz) = \phi(y) \cap \phi(z) = \{ T_1\}$. Since $T_3 \in \phi(wz)$ and $T_2 \in \phi(x y)$, we may choose $u, v \in V(H) \setminus \{x,y,z,w\}$ for which $\phi(uwz) = T_3$ and $\phi(vxy) = T_2$. This gives the structure shown in Figure~\ref{fig: two tight comp}.
    \begin{figure}[!t]
        \centering
        \includegraphics[width=0.45\linewidth]{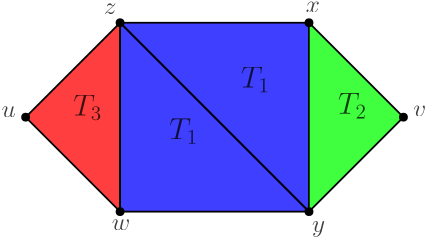}
        \caption{Edges and components in Lemma~\ref{Lma: two tight components}}
        \label{fig: two tight comp}
    \end{figure}

    Let $\mathcal{P} = \bigcup_{e \in \{vxy, xyz, wyz, uwz\}} \partial e$, so $| \mathcal{P} | = 9$. Since $\phi(xyz) \neq \phi(vxy)$ and $\phi(uwz) \neq \phi(wyz)$, by Fact~\ref{fct: common nbr of adjacent triangles} we have 
    $$\bigcap_{P \in \partial xyz \cup \partial vxy} N_H (P) = \emptyset \mbox{ and } \bigcap_{P \in \partial wyz \cup \partial uwz} N_H (P) = \emptyset.$$ 
    On the other hand, by Proposition~\ref{Prp: 8 from 9 and 7 with one fixed}\ref{itm: i - 8 pairs common nbr} there exists $P' \in \mathcal{P}$ with $\bigcap_{P \in \mathcal{P} \setminus P'} N_H(P) \neq \emptyset$. So we must have $P' = yz$, and therefore may choose $p \in \bigcap_{P \in \mathcal{P} \setminus \{yz\}} N_H(P)$. Since $puwz$ and $pvxy$ are edges of $\T(H)$ we have $\phi(puw) = \phi(uwz) = T_3$ and $\phi(pvx) = \phi(vxy) = T_2$, so~$\phi(p) = \{T_2, T_3\}$ by Lemma~\ref{lma: every vtx sees 2 colours}. So~$\phi(x) = \{T_1, T_2\},\ \phi(p) = \{T_2, T_3\}$ and~$\phi(z) = \{T_1, T_3\}$, contradicting Proposition~\ref{prop: no alternating triangle}.  
\end{proof}

\subsection{There is only one tight component} \label{subsec: one tight comp}
In this section we complete the proof of Lemma~\ref{lma: tight connectivity at 7n/9}. By Lemma~\ref{Lma: two tight components} we know that $H$ has only two tight components, which we call red and blue (denoted $r$ and $b$), so each edge of $H$ is coloured either red or blue. We now give several results describing certain colour patterns that cannot occur in~$H$. For edge-coloured $3$-graphs $G$ and $H$, when we write `$G$ is not a subgraph of $H$', we mean that $G$ is not a colour-preserving subgraph of $H$. 

\begin{proposition}\label{prp: no colour balanced double tetrahedron}
    Let~$H$ be a~$3$-graph on~$n$ vertices with~$\delta_2(H) > 7n/9$, and let $r, b$ be distinct tight components of~$\T(H)$. Let~$G$ be a~$3$-graph with vertices~$x, y_1, y_2, y_3, z$ and edge set $\{xy_i y_{i+1}, zy_i y_{i+1} \colon i \in [3]\}$, in which $3$ edges have colour $r$ and $3$ edges have colour $b$. Then~$G$ is not a subgraph of~$H$.
\end{proposition}

\begin{proof}
Suppose for a contradiction that $H$ contains a copy of~$G$. Let $\mathcal{P}$ be the set of all pairs which are edges of $\partial G$, so $|\mathcal{P}| = 9$. By Proposition~\ref{Prp: 8 from 9 and 7 with one fixed}\ref{itm: i - 8 pairs common nbr} there exists $P' \in \mathcal{P}$ for which $\bigcap_{P \in \mathcal{P} \setminus P'} N_H(P) \neq \emptyset$.

Without loss of generality we may assume that $x$ is contained in at least two red edges. So either~$x$ is in three red edges, in which case we may assume without loss of generality we are in Case 1 below, or~$x$ is in precisely two red edges, in which case we may assume without loss of generality we are in either Case 2 or Case 3 below, according to whether the intersection of the blue edge containing $x$ and the red edge containing $z$ has size one or two.    

\medskip \noindent\textbf{Case $1$:} $\phi(xy_1y_2) = \phi(xy_1y_3) = \phi(xy_2y_3) = r$ and $y_3 \in P'$. In this case, our choice of $P'$ ensures that $\bigcap_{P \in \partial xy_1y_2 \cup \partial zy_1y_2} N_H(P) \neq \emptyset$; since $\phi(zy_1y_2) = b$ this contradicts Fact~\ref{fct: common nbr of adjacent triangles}. 
    
\medskip \noindent\textbf{Case $2$:} $\phi(xy_1y_2) = \phi(xy_2y_3) = \phi(zy_1y_3) = r$ and $\phi(zy_1y_{2}) = \phi(zy_2y_3) = \phi(xy_1y_3) = b$. In this case $xy_1zy_2$ is a properly coloured $C_4$ in~$L(y_3)$, contradicting Lemma~\ref{lma: link graph properties}\ref{itm: i - no properly col C4}.  

\medskip \noindent\textbf{Case $3$:} $\phi(xy_1y_{2}) = \phi(xy_2y_3) = \phi(zy_1y_2) = r$ and~$\phi(xy_1y_3) = \phi(zy_2y_3) = \phi(zy_1y_3)=b$. In this case set 
$    \mathcal{M} = \{  (xy_1y_{2},xy_1y_3)  ,   (xy_2y_{3},zy_2y_3) , (zy_1y_{2},zy_2y_3)  \}$.
Observe that no pair $P \in \mathcal{P}$ has $P \in \partial e \cup \partial e'$ for all three pairs $(e, e') \in \mathcal{M}$. It follows that some pair $(e, e') \in \mathcal{M}$ has $\bigcap_{P \in \partial e \cup \partial e'} N_H(P) \neq \emptyset$. Since each pair $(e, e') \in \mathcal{M}$ has $|e \cap e'| = 2$ and  $\phi(e) = r  \ne b =  \phi(e')$, this contradicts Fact~\ref{fct: common nbr of adjacent triangles}.
\end{proof}
    
In the next lemma, we show that $H$ cannot contain a tight walk of length 4 with alternating colours (see Figure~\ref{fig: alt tight path length 4} for an illustration).

\begin{figure}[!t]
        \centering
        \includegraphics[width=0.5\linewidth]{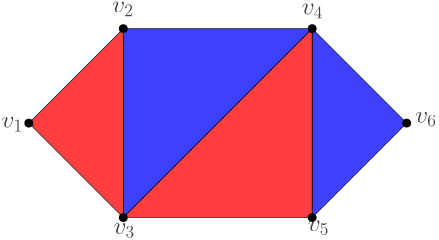}
        \caption{A tight path $v_1v_2\dots v_6$ with~$\phi(v_1v_2v_3) = \phi(v_3v_4v_5) = r$ and~$\phi(v_2v_3v_4) = \phi(v_4v_5v_6) = b$ mentioned in Lemma~\ref{lma: no alt tight path length 4}}
        \label{fig: alt tight path length 4}
    \end{figure}

\begin{lemma} \label{lma: no alt tight path length 4}
    Let~$H$ be a~$3$-graph on~$n$ vertices with~$\delta_2(H) >7n/9$  and let~$r, b$ be distinct tight components of~$\T(H)$. Let~$\mathscr{P} = v_1v_2\dots v_6$ be a tight walk such that~$\phi(v_1v_2v_3) = \phi(v_3v_4v_5) = r$ and~$\phi(v_2v_3v_4) = \phi(v_4v_5v_6) = b$. Then~$\mathscr{P}$ is not a subgraph of~$H$.
\end{lemma}

\begin{proof}
Suppose for a contradiction that $H$ contains a copy of~$\mathscr{P}$. Observe that the vertices $v_2, v_3, v_4$ and $v_5$ must be distinct since $v_2v_3v_4$ is a blue edge and $v_3v_4v_5$ is a red edge. Also, since $\phi(v_1v_2v_3) = r$ we have $r \in \phi(v_2v_3)$ and so by Fact~\ref{fct: tetrahedral codegree} there are at least $n/3$ vertices $u_1$ for which $u_1v_2v_3$ is a red edge of~$H$. In the same way, since $\phi(v_4v_5v_6)$ is a blue edge of $H$ there are at least $n/3$ vertices $u_6$ for which $v_4v_5u_6$ is a blue edge of $H$. By choosing such $u_1$ and $u_6$ to be distinct from each other and from $v_2, v_3, v_4$ and $v_5$, we obtain a tight path $u_1v_2v_3v_4v_5u_6$ in $H$ with~$\phi(u_1v_2v_3) = \phi(v_3v_4v_5) = r$ and~$\phi(v_2v_3v_4) = \phi(v_4v_5u_6) = b$. We therefore assume without loss of generality that $\mathscr{P}$ itself is a tight path, that is, that $v_1,v_2,v_3,v_4,v_5$ and $v_6$ are all distinct.

Let $\mathcal{P}$ be the set of pairs which are edges of $\partial \mathscr{P}$, so~$|\mathcal{P}| = 9$. By Proposition~\ref{Prp: 8 from 9 and 7 with one fixed}\ref{itm: i - 8 pairs common nbr}, there exists~$P' \in \mathcal{P}$ such that $\bigcap_{P \in \mathcal{P} \setminus P'} N_H(P) \neq \emptyset$. If~$P' \neq v_3v_4$, then $\bigcap_{P \in \partial e \cup \partial e'} N_H(P) \neq \emptyset$ for some 
$$(e, e') \in \{(v_1v_2v_3, v_2v_3v_4), (v_3v_4v_5, v_4v_5v_6)\}.$$ 
By Fact~\ref{fct: common nbr of adjacent triangles} it follows that $e$ and $e'$ are in the same tight component, contradicting our assumption that $\phi(e) = r$ and $\phi(e') = b$. 

We therefore assume that $P' = v_3v_4$. Choose~$v \in \bigcap_{P \in \mathcal{P} \setminus P'} N_H(P)$. Then~$vv_1v_2v_3$ and $vv_4v_5v_6$ are copies of~$K_4^3$ in $H$, so $\phi(vv_2v_3) = \phi(v_1v_2v_3) = r$ and $\phi(vv_4v_5) = \phi(v_4v_5v_6) = b$.  
If~$\phi(vv_2v_4) = r$ (or $\phi(vv_3v_5) = b$), then~$vv_2v_3v_5$ (or $vv_2v_4v_5$) forms a properly coloured copy of~$C_4$ in~$L(v_4)$ (or in~$L(v_3)$ respectively), contradicting Lemma~\ref{lma: link graph properties}\ref{itm: i - no properly col C4}. So we must have~$\phi(vv_2v_4) = b$ and~$\phi(vv_3v_5) = r$, and so 
$$\phi(vv_2v_3) = \phi(v_3v_4v_5) = \phi(vv_3v_5) = r \text{ and } \phi(vv_4v_5) = \phi(vv_2v_4) = \phi(v_2v_3v_4) = b,$$ giving a contradiction to Proposition~\ref{prp: no colour balanced double tetrahedron} with~$v_2, v_5, v, v_3$ and~$v_4$ playing the roles of~$x, z, y_1, y_2$ and~$y_3$ respectively.\end{proof}

The next lemma forbids an edge-coloured tight walk of length~$5$ (see Figure~\ref{fig: weird tight path of length 5} for an illustration).

    \begin{figure}[!t]
        \centering
        \includegraphics[width=0.6\linewidth]{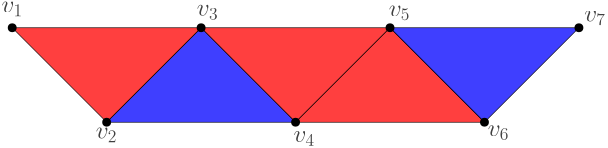}
        \caption{
        A tight path $v_1v_2\dots v_7$ with~$\phi(v_1v_2v_3) = \phi(v_3v_4v_5) = \phi(v_4v_5v_6)= r$ and~$\phi(v_2v_3v_4) = \phi(v_5v_6v_7) = b$
        mentioned in Lemma~\ref{lma: no weird tight path length 5}}
        \label{fig: weird tight path of length 5}
    \end{figure}

\begin{lemma} \label{lma: no weird tight path length 5}
    Let~$H$ be a~$3$-graph on~$n$ vertices such that~$\delta_2(H) >7n/9$ and let~$r, b$ be distinct tight components of~$\T(H)$. Let~$\mathscr{P} = v_1v_2\dots v_7$ be a tight walk such that~$\phi(v_1v_2v_3) = \phi(v_3v_4v_5) = \phi(v_4v_5v_6)= r$ and~$\phi(v_2v_3v_4) = \phi(v_5v_6v_7) = b$. Then~$\mathscr{P}$ is not a subgraph of~$H$. 
\end{lemma}

\begin{proof}
Suppose for a contradiction that $H$ contains a copy of~$\mathscr{P}$. Again the vertices $v_2, v_3, v_4$ and $v_5$ must be distinct. Since $\phi(v_5v_6v_7) = b$ we have $|N_b(v_5v_6)| \geq n/3$ by Fact~\ref{fct: tetrahedral codegree}, so there are at least $n/3 - 2n/9 = n/9$ choices for $y \in N(v_4v_5)  \cap  N_b(v_5 v_6)$. By Lemma~\ref{lma: no alt tight path length 4} we must have $\phi( v_4,v_5 ) = \{r\}$, so each such $y$ has $\phi(v_4 v_5 y) = r$ and $ \phi (v_5 v_6 y)=b$. Having chosen $y$, by Fact~\ref{fct: tetrahedral codegree} there are at least $n/9$ and $n/3$ choices for each of $z \in N_b(v_5 y ) \cap N_H(v_4v_5)$ and $x \in N_{r}(v_2v_3)$ respectively, so we may choose such $x, y$ and $z$ to be distinct from each other and from $v_2, v_3, v_4$ and $v_5$, giving a tight path $xv_2v_3 v_4 v_5yz$ in $H$ with $\phi(xv_2v_3) = \phi(v_3v_4v_5) = \phi(v_4v_5y)= r$ and~$\phi(v_2v_3v_4) = \phi(v_5yz) = b$ for which $v_4v_5z$ is an edge of $H$ with $\phi(v_4v_5z) = r$. 
We therefore assume without loss of generality that $\mathscr{P}$ itself is a tight path (that is, that $v_1,v_2,v_3,v_4,v_5, v_6$ and $v_7$ are all distinct) and also that $v_4v_5v_7 \in E(H)$ with $\phi(v_4v_5v_7 )= r$.

Let $\mathcal{P}$ be the set of pairs which are edges of $\partial \mathscr{P} \setminus \{v_1\}$, so~$|\mathcal{P}| = 9$. By Proposition~\ref{Prp: 8 from 9 and 7 with one fixed}\ref{itm: i - 8 pairs common nbr}, there exists $P' \in \mathcal{P}$ such that $\bigcap_{P \in \mathcal{P} \setminus P'} N_H(P) \neq \emptyset$. If~$P' \neq v_4v_5$, then $\bigcap_{P \in \partial e \cup \partial e'} N_H (P) \neq \emptyset$ for some 
$$(e, e') \in \{(v_3v_4v_5,v_2v_3v_4), (v_4v_5v_6, v_5v_6v_7)\}.$$
By Fact~\ref{fct: common nbr of adjacent triangles} it follows that $e$ and $e'$ are in the same tight component, contradicting our assumption that $\phi(e) = r$ and $\phi(e') = b$.

We therefore assume that~$P' = v_4v_5$. Choose~$x \in \bigcap_{P \in \mathcal{P} \setminus P'} N_H(P)$. Then $xv_5v_6v_7$ and~$xv_2v_3v_4$ are copies of $K_4^3$ in $H$, so $\phi(xv_5v_6) = \phi(v_5v_6v_7) = b$ and $\phi(xv_3v_4) = \phi(v_2v_3v_4) = b$.
Since $v_4v_6 \in \mathcal{P}$ we know that $xv_4x_6 \in E(H)$. So if $\phi(xv_4v_6) = r$, then $v_4 v_5 v_6 x v_4 v_3$ is a tight walk in~$H$ with $\phi(v_4v_5v_6) = \phi(xv_4v_6) = r$ and~$\phi(xv_5v_6) = \phi(xv_3v_4) = b$. On the other hand, if $\phi(xv_4v_6) = b$, then $ v_4 v_7 v_5 v_6 v_4 x$ is a tight walk in~$H$ with $\phi(v_4v_5v_7) = \phi(v_4v_5v_6) = r$ and~$\phi(v_5v_6v_7) = \phi(x v_4v_6) = b$. In either case the tight walk obtained gives a contradiction to Lemma~\ref{lma: no alt tight path length 4}.
\end{proof}

By combining the previous two lemmas we obtain the following proposition.

\begin{proposition}\label{z is red}
    Let~$H$ be a~$3$-graph on~$n$ vertices with~$\delta_2(H) >7n/9$ such that $\T(H)$ has precisely two distinct tight components, $r$ and $b$. If $xyz, wxy, wyz \in E(H)$ are such that $\phi(xyz)=\phi(wyz)=r$ and $\phi(wxy)=b$, then $\phi(z)=\{r\}$.
\end{proposition}    

\begin{proof}
Suppose for a contradiction that $b \in \phi(z)$, so we may choose $v \in V(H)$ with $b \in \phi(vz)$. By Fact~\ref{fct: tetrahedral codegree} we may then choose $v' \in N_{ b}(vz) \cap N_H(xz)$, and in particular we have $\phi(zvv') = b$. On the other hand, we must have $\phi(xzv')=r$ as otherwise $zwyxzv'$ would be a tight walk on $6$ vertices which contradicts Lemma~\ref{lma: no alt tight path length 4}. So $zwyxzv'v$ is a tight walk whose colouring contradicts Lemma~\ref{lma: no weird tight path length 5}. We conclude that $b \notin \phi(z)$, so $\phi(z) = \{r\}$.
\end{proof}

We say that a tight component is \emph{spanning} in~$H$ if every vertex is in an edge of that tight component. By Lemma~\ref{Lma: two tight components} we know that $\T(H)$ has exactly two tight components,~$r$ and~$b$, and Fact~\ref{fct: colours are intersecting family} implies that at least one of them is spanning, say~$r$. 
In the next lemma, we show that $b$ is not spanning and study its properties.

\begin{lemma} \label{lma: some vts see only red and blue K_5}
    Let~$H$ be a~$3$-graph on~$n$ vertices with~$\delta_2(H) > 7n/9$, let~$r, b$ be distinct tight components of~$\T(H)$, and suppose that $r$ is spanning. Then the following statements hold. 
    \begin{enumerate}[label={\rm(\roman*)}]
    \item $b$ is not spanning. \label{itm: red spanning}
    \item Every $xy \in \binom{V(H)}{2}$ with $\phi(xy)=\{r, b\}$ has $N_{r}(xy) \subseteq U:= \{z \in V(H) \colon  \phi(z)=\{r\}\}$.\label{itm: many red vtx}
    \item There is a copy of $K^3_5$ in $H$ all of whose edges are coloured $b$. \label{itm: blue K5}
    \end{enumerate}
\end{lemma}

\begin{proof}
By Lemma~\ref{Lma: two tight components}, $r$ and $b$ are the only tight components of $\T(H)$.
Fix $xy \in \binom{V(H)}{2}$ with $\phi(xy) = \{r, b\}$ and $z \in N_{r}(xy)$, so in particular we have $xyz \in E(H)$ and $\phi(xyz) = r$. Proposition~\ref{prp: weird configurations exist} ensures that such $x$ and $y$ exist, and also that there exists $w \in V(H)$ with $wxy, wyz \in E(H)$ and $\phi(wxy)=b$. If $\phi(wyz) = b$ then by  Proposition~\ref{z is red} (with the roles of $r$ and $b$ swapped, as are the roles of $z$ and $w$) we have $\phi(w)=\{b\}$, contradicting our assumption that $r$ is spanning. So we must have $\phi(wyz) = r$, so by Proposition~\ref{z is red} we have $\phi(z) = \{r\}$. This gives~\ref{itm: red spanning} immediately, and also~\ref{itm: many red vtx} because $z \in N_{r}(xy)$ was arbitrary.

For~\ref{itm: blue K5}, observe that since $\phi(wxy) = b$, by Fact~\ref{fct: tetrahedral codegree} we may choose $u \in N_b(xy) \cap N_b(wx) \cap N_b(wy)$, and then $uwxy$ forms a copy of $K^3_4$ in $H$ whose edges are all coloured $b$.
Note that~$|\partial wyz  \cup \partial xyz \cup \partial^2 uwxy|=9$ (where $\partial^2uwxy = \{uw, ux, uy, wx, wy, xy\}$).  So by Proposition~\ref{Prp: 8 from 9 and 7 with one fixed}\ref{itm: i - 8 pairs common nbr} there exists $P' \in  \partial wyz \cup \partial xyz \cup \partial^2 uwxy$ such that $$\bigcap_{P \in \left(\partial wyz \cup \partial xyz \cup \partial^2 uwxy\right) \setminus P'} N_H (P) \neq \emptyset.$$  
If $P' \neq yz$, then for some $$(e, e') \in \{(xyz, uxy), (xyz, wxy), (wyz, uwy), (wyz, wxy)\},$$ we have $\bigcap_{P \in \partial e \cup \partial e'} N_H(P) \neq \emptyset$, giving a contradiction to Fact~\ref{fct: common nbr of adjacent triangles} because $\phi(e) = r$ and $\phi(e') = b$. So we must have $P' = yz$. Then there exists a vertex $u' \in \bigcap_{P \in \partial^2uwxy}N_H(P)$, implying that $uu'wxy$ forms a copy of $K_5^3$ in $H$. Each set of four vertices in this copy forms a copy of $K^3_4$ in $H$ which shares at least three vertices with $uwxy$ and so lies in the same tight component $b$ of $\T(H)$, that is, $b$. It follows that all edges of~$uu'wxy$ have colour~$b$.  
\end{proof}

We are finally ready to prove Lemma~\ref{lma: tight connectivity at 7n/9}. 

\begin{proof}[{Proof of Lemma~\ref{lma: tight connectivity at 7n/9}}]
Suppose for a contradiction that $\T(H)$ is not tightly connected.
By Lemma~\ref{Lma: two tight components}, there are then exactly two components of~$\T(H)$, say~$r$ and~$b$, and by Fact~\ref{fct: colours are intersecting family} and~Lemma~\ref{lma: some vts see only red and blue K_5}\ref{itm: red spanning} we assume without loss of generality that $r$ is spanning and $b$ is not.
Let $U := \{z \in V(H) \colon  \phi(z)= \{r\}\}$; by combining 
Proposition~\ref{prp: weird configurations exist}, Fact~\ref{fct: tetrahedral codegree} and Lemma~\ref{lma: some vts see only red and blue K_5}\ref{itm: many red vtx} we then obtain $|U| > n/3$.
Furthermore, by Lemma~\ref{lma: some vts see only red and blue K_5}\ref{itm: blue K5}, there is a copy of $K_5^3$ in $H$ all of whose edges are coloured $b$; let $y_1, \dots, y_5$ be the vertices of this copy. 
For all distinct~$i, j \in [5]$ we have \begin{equation} \label{eqn: nbrs of pairs in U}
    |N_H (y_i y_j) \cap U| \geq \delta_2(H) + |U| - n > |U| - 2n/9 > 0,
\end{equation} 
so we may choose $u_{ij} \in N_H (y_i y_j) \cap U$.
By definition of~$U$ we then have $r = \phi(u_{ij}y_iy_j) \in \phi(y_iy_j)$, so $\phi(y_iy_j) = \{r, b\}$. 
By Lemma~\ref{lma: some vts see only red and blue K_5}\ref{itm: many red vtx} and Fact~\ref{fct: tetrahedral codegree} it follows that $|N_H(y_i y_j) \cap U| = |N_{r}(y_iy_j)| > n/3$. Together with~\eqref{eqn: nbrs of pairs in U} this yields \begin{align*}
    \sum_{ij \in \binom{[5]}{2}} |N_H(y_i y_j)\cap U| &> 6\left(|U| - \frac{2n}{9}\right) + \frac{4n}{3} = 6|U|.
\end{align*}
By averaging over~$U$ it follows that some $u \in U$ is in $N_H(y_i y_j)$ for at least seven pairs $ij \in \binom{[5]}{2}$. Applying Mantel's theorem we obtain a triangle $ijk$ in the graph on vertex set $[5]$ whose edges are these pairs. Then $uy_iy_jy_k$ is a copy of~$K_4^3$ in~$H$ with $\phi(y_iy_jy_k) = b$ and $\phi(uy_iy_j)=r$ (because $u \in U$), giving a contradiction.
\end{proof}

\section{Hypergraph Regularity} \label{sec: Hypergraph Regularity}
In this section, we formulate the notion of hypergraph regularity that we use, closely following the formulation from Allen, B\"ottcher, Cooley and Mycroft~\cite{ABCM}.
Recall that a \emph{hypergraph}~$\Hy$ is an ordered pair~$(V(\Hy), E(\Hy))$, where~$E(\Hy) \subseteq 2^{V(\Hy)}$. 
We identify the hypergraph~$\Hy$ with its edge set~$E(\Hy)$. 
A subgraph~$\Hy'$ of~$\Hy$ is a hypergraph with~$V(\Hy') \subseteq V(\Hy)$ and~$E(\Hy') \subseteq E(\Hy)$. 
It is \emph{spanning} if~$V(\Hy') = V(\Hy)$.
For~$U \subseteq V(\Hy)$, we define~$\Hy[U]$ to be the subgraph of~$\Hy$ with~$V(\Hy[U]) = U$ and~$E(\Hy[U]) = \{ e \in E(\Hy) \colon e \subseteq U\}$.

A hypergraph~$\Hy$ is called a \emph{complex} if~$\Hy$ is down-closed, that is if for an edge~$e \in \Hy$ and~$f \subseteq e$, then~$f \in \Hy$.
A \emph{$k$-complex} is a complex having only edges of size at most~$k$. 
We denote by~$\Hy^{(i)}$ the spanning subgraph of~$\Hy$ containing only the edges of size~$i$. Given a $k$-complex $\Hy$ and a set of vertices $S \subseteq V(\Hy)$, the subcomplex induced on~$S$ is obtained by taking all edges of~$\Hy$ induced on~$S$.
Let~$\mathcal{P}$ be a partition of~$V(\Hy)$ into vertex classes~$V_1, \dots, V_s$. Then we say that a set~$S \subseteq V(\Hy)$ is \emph{$\mathcal{P}$-partite} if~$\s{S \cap V_i} \leq 1$ for all~$i \in [s]$. 
For~$\mathcal{P}' = \{V_{i_1}, \dots, V_{i_r}\} \subseteq \mathcal{P}$, we define the subgraph of~$\Hy$ induced by~$\mathcal{P}'$, denoted by~$\Hy[\mathcal{P'}]$ or~$\Hy[V_{i_1}, \dots, V_{i_r}]$, to be the subgraph of~$\Hy[\bigcup \mathcal{P}']$ containing only the edges that are~$\mathcal{P}'$-partite. 
The hypergraph~$\Hy$ is said to be~$\mathcal{P}$-partite if all of its edges are~$\mathcal{P}$-partite. We say that~$\Hy$ is \emph{$s$-partite} if it is~$\mathcal{P}$-partite for some partition~$\mathcal{P}$ of~$V(\Hy)$ into~$s$ parts.
Let~$\Hy$ be a~$\mathcal{P}$-partite hypergraph. 
If~$X$ is a~$k$-set of vertex classes of~$\Hy$, then we write~$\Hy_X$ for the~$k$-partite subgraph of~$\Hy^{(k)}$ induced by~$\bigcup X$, whose vertex classes are the elements of~$X$. 
Moreover, we denote by~$\Hy_{X^<}$ the~$k$-partite hypergraph with~$V(\Hy_{X^<}) = \bigcup X$ and~$E(\Hy_{X^<}) = \bigcup_{X'\subsetneq X} \Hy_{X'}$.
In particular, if~$\Hy$ is a complex, then~$\Hy_{X^<}$ is a~$(k-1)$-complex because~$X$ is a set of size~$k$.

Let~$i \geq 2$ and let~$\mathcal{P}_i$ be a partition of a vertex set~$V$ into~$i$ parts. Let~$H_i$ and~$H_{i-1}$ be a~$\mathcal{P}_i$-partite~$i$-graph and a~$\mathcal{P}_i$-partite~$(i-1)$-graph on a common vertex set~$V$, respectively. We say that a~$\mathcal{P}_i$-partite~$i$-set in~$V$ is \emph{supported on}~$H_{i-1}$ if it induces a copy of the complete~$(i-1)$-graph~$K_i^{(i-1)}$ on~$i$ vertices in~$H_{i-1}$. We denote by~$K_i(H_{i-1})$ the~$\mathcal{P}_i$-partite~$i$-graph on~$V$ whose edges are all~$\mathcal{P}_i$-partite~$i$-sets contained in~$V$ which are supported on~$H_{i-1}$. Now we define the \emph{density of~$H_i$ with respect to~$H_{i-1}$} to be
\[
d(H_i \mid H_{i-1}) = \frac{\s{K_i(H_{i-1}) \cap H_i}}{\s{K_i(H_{i-1})}}
\]
if~$\s{K_i(H_{i-1})} > 0$ and~$d(H_i \mid H_{i-1}) = 0$ if~$\s{K_i(H_{i-1})} = 0$.
So~$d(H_i \mid H_{i-1})$ is the proportion of~$\mathcal{P}_i$-partite copies of~$K_i^{i-1}$ in~$H_{i-1}$ which are also edges of~$H_i$. More generally, if~$\mathbf{Q} = (Q_1, Q_2, \dots, Q_r)$ is a collection of~$r$ (not necessarily disjoint) subgraphs of~$H_{i-1}$, we define~$K_i(\mathbf{Q}) = \bigcup_{j=1}^r K_i(Q_j)$ and 
\[
d(H_i \mid \mathbf{Q}) = \frac{\s{K_i(\mathbf{Q}) \cap H_i}}{\s{K_i(\mathbf{Q})}}
\]
if~$\s{K_i(\mathbf{Q})} > 0$ and~$d(H_i \mid \mathbf{Q}) = 0$ if~$\s{K_i(\mathbf{Q})} = 0$.
We say that~$H_i$ is \emph{$(d_i, \eps, r)$-regular with respect to~$H_{i-1}$}, if we have~$d(H_i \mid \mathbf{Q}) = d_i \pm \eps$ for every~$r$-set~$\mathbf{Q}$ of subgraphs of~$H_{i-1}$ with~$\s{K_i(\mathbf{Q})} > \eps \s{K_i(H_{i-1})}$.
We say that~$H_i$ is \emph{$(\eps, r)$-regular with respect to~$H_{i-1}$} if there exists some~$d_i$ for which~$H_i$ is~$(d_i, \eps, r)$-regular with respect to~$H_{i-1}$. 
Finally, given an~$i$-graph~$G$ whose vertex set contains that of~$H_{i-1}$, we say that~$G$ is \emph{$(d_i, \eps, r)$-regular with respect to~$H_{i-1}$} if the~$i$-partite subgraph of~$G$ induced by the vertex classes of~$H_{i-1}$ is~$(d_i, \eps, r)$-regular with respect to~$H_{i-1}$. 
We refer to the density of this~$i$-partite subgraph of~$G$ with respect to~$H_{i-1}$ as the \emph{relative density of~$G$ with respect to~$H_{i-1}$}.

Now let~$s \geq k \geq 3$ and let~$\Hy$ be an~$s$-partite~$k$-complex on vertex classes~$V_1, \dots, V_s$. For any set~$A \subseteq [s]$, we write~$V_A$ for~$\bigcup_{i \in A} V_i$. Note that, if~$e \in \Hy^{(i)}$ for some~$2 \leq i \leq k$, then the vertices of~$e$ induce a copy of~$K_i^{i-1}$ in~$\Hy^{(i-1)}$. Therefore, for any set~$A \in \binom{[s]}{i}$, the density~$d(\Hy^{(i)}[V_A] \mid \Hy^{(i-1)}[V_A])$ is the proportion of `possible edges' of~$\Hy^{(i)}[V_A]$, which are indeed edges. We say that~$\Hy$ is \emph{$(d_k, \dots, d_2, \eps_k, \eps, r)$-regular} if
\begin{enumerate}[label=(\alph*)]
    \item for any~$2 \leq i \leq k-1$ and any~$A \in \binom{[s]}{i}$, the induced subgraph~$\Hy^{(i)}[V_A]$ is~$(d_i, \eps, 1)$-regular with respect to~$\Hy^{(i-1)}[V_A]$ and 
    \item for any~$A \in \binom{[s]}{k}$, the induced subgraph~$\Hy^{(k)}[V_A]$ is~$(d_k, \eps_k, r)$-regular with respect to~$\Hy^{(k-1)}[V_A]$.
\end{enumerate}
For~$\mathbf{d} = (d_k, \dots, d_2)$, we write~$(\mathbf{d}, \eps_k, \eps, r)$-regular to mean~$(d_k, \dots, d_2, \eps_k, \eps, r)$-regular.

A key property of regular complexes is that the restriction of such a complex to a large subset of its vertex set is also a regular complex, with the same relative densities at each level of the complex, although with somewhat weakened regularity properties. The next lemma states this property formally. 

\begin{lemma}[Regular Restriction Lemma~{\cite[Lemma 28]{ABCM}}]\label{lma: regular restriction} 
Suppose integers~$k, m$ and reals $\alpha,\eps,\eps_k$, and $d_2, \dots, d_k > 0$ are such that
$$1/m \ll \eps\ll \eps_k,d_2,\dots,d_{k-1}\quad\text{and}\quad  \eps_k\ll \alpha,1/k\,.$$ 
For any~$r,s\in\N$ and~$d_k>0$, set~$\mathbf{d} = (d_k, \dots, d_2)$ and let~$\cG$ be an~$s$-partite~$k$-complex whose vertex classes~$V_1, \dots, V_s$ each have size~$m$ and which is~$(\mathbf{d}, \eps_k, \eps, r)$-regular. Choose any~$V'_i \subseteq V_i$ with~$|V'_i| \geq \alpha m$ for each~$i \in [s]$. Then the induced subcomplex~$\cG[V'_1 \cup \dots \cup V'_s]$ is~$(\mathbf{d},\sqrt{\eps_k},\sqrt{\eps},r)$-regular. 
\end{lemma}

We say that a~$(k-1)$-complex~$\J$ is \emph{$(t_0, t_1, \eps)$-equitable} if it has the following properties.
\begin{enumerate}[label = (\alph*)]
    \item~$\J$ is~$\mathcal{P}$-partite for some~$\mathcal{P}$ which partitions~$V(\J)$ into~$t$ parts, where~$t_0 \leq t \leq t_1$, of equal size. We refer to~$\mathcal{P}$ as the \emph{ground partition} of~$\J$ and to the parts of~$\mathcal{P}$ as the \emph{clusters} of~$\J$.
    \item There exists a \emph{density vector}~$\mathbf{d} = (d_{k-1}, \dots, d_2)$ such that, for each~$2 \leq i \leq k-1$, we have~$d_i \geq 1/t_1$ and~$1/d_i \in \N$ and~$\J$ is~$(\mathbf{d}, \eps, \eps, 1)$-regular.
\end{enumerate}
For any~$k$-set~$X$ of clusters of~$\J$, we denote by~$\hat{\J}_X$ the~$k$-partite~$(k-1)$-graph~$(\J_{X^<})^{(k-1)}$ and call~$\hat{\J}_X$ a \emph{polyad}. Given a~$(t_0, t_1, \eps)$-equitable~$(k-1)$-complex~$\J$ and a~$k$-graph~$G$ on~$V(\J)$, we say that~$G$ is \emph{$(\eps_k, r)$-regular with respect to a~$k$-set~$X$ of clusters of~$\J$} if there exists some~$d$ such that~$G$ is~$(d, \eps_k, r)$-regular with respect to the polyad~$\hat{\J}_X$.
Moreover, we write~$d_{G, \J}^*(X)$ for the relative density of~$G$ with respect to~$\hat{\J}_X$; we may drop either subscript if it is clear from context.

We can now give the crucial definition of a regular slice.
\begin{definition}[Regular slice]
Given~$\eps, \eps_k > 0, r, t_0, t_1 \in \N$, a~$k$-graph~$G$, a~$(k-1)$-complex~$\J$ on~$V(G)$, is a \emph{$(t_0, t_1, \eps, \eps_k,r)$-regular slice} for~$G$ if~$\J$ is~$(t_0, t_1, \eps)$-equitable and~$G$ is~$(\eps_k, r)$-regular with respect to all but at most~$\eps_k \binom{t}{k}$ of the~$k$-sets of clusters of~$\J$, where~$t$ is the number of clusters of~$\J$.
\end{definition}

If we specify the density vector~$\mathbf{d}$ and the number of clusters~$t$ of an equitable complex or a regular slice, then it is not necessary to specify~$t_0$ and~$t_1$ (since the only role of these is to bound~$\mathbf{d}$ and~$t$). In this situation we write that~$\J$ is~$(\cdot, \cdot, \eps)$-equitable, or is a~$(\cdot, \cdot, \eps, \eps_k, r)$-regular slice for~$G$.

Given a regular slice~$\J$ for a~$k$-graph~$G$, it will be important to know the
relative densities~$d^*(X)$ for~$k$-sets~$X$ of clusters of~$\J$. To keep track of these we use the following definition.

\begin{definition}[Weighted reduced~$k$-graph]
 Given a~$k$-graph~$G$ and a~$(t_0, t_1, \eps)$-equitable
~$(k-1)$-complex~$\J$ on~$V(G)$, we let~$R_\J(G)$ be the complete weighted~$k$-graph whose
 vertices are the clusters of~$\J$ and where each edge~$X$ is given weight~$d^*(X)$
 (in particular, the weight is in~$[0,1]$). When~$\J$ is clear from the context we often simply write~$R(G)$ instead of~$R_\J(G)$. 
\end{definition}

Given a set~$S \subseteq V(G)$ of size~$j$ for some~$j \in [k-1]$, the \emph{relative degree~$\overline{\deg}(S; G)$ of~$S$ with respect to~$G$} is defined as 
$$\overline{\deg}(S; G) : = \frac{|\{ e \in E(G) \colon S \subseteq e\}|}{\binom{| V(G)\setminus S |}{k-j}}.$$
Similarly, if~$G$ is a weighted~$k$-graph with weight function~$d^*$, then we define 
$$\overline{\deg}(S; G) = \frac{\sum_{e \in E(G) \colon S \subseteq e} d^*(e) }{\binom{|V(G)\setminus S|}{k-j}}.$$
Given a $k$-graph $G$ and distinct `root' vertices
$v_1,\ldots,v_\ell$ of~$G$ and a $k$-graph $H$ equipped with a set of distinct
`root' vertices $x_1,\ldots,x_\ell$, the \emph{number of labelled rooted copies of~$H$ in~$G$}, denoted by~$n_H(G;v_1,\ldots,v_\ell)$, is defined to be the number of injective maps from~$V(H)$ to~$V(G)$ embedding~$H$ in~$G$ and taking~$x_j$ to~$v_j$ for each~$j \in [\ell]$. Then the \emph{density of rooted copies of~$H$ in~$G$} is defined as
\[d_H(G;v_1,\ldots,v_\ell):=\frac{n_H(G;v_1,\ldots,v_\ell)}{\tbinom{|V(G)|-\ell}{|V(H)|-\ell}\cdot\big(|V(H)|-\ell\big)!}\,.\]
This density has a natural probabilistic interpretation: choose uniformly at random an injective map $\psi : V(H) \to V(G)$ such that $\psi(x_j) = v_j$ for each~$j \in [\ell]$. Then $d_H(G;v_1, \dots, v_\ell)$ is the probability that $\psi$ embeds $H$ into~$G$. 
We now define $\Hskel$ to be the $(k-1)$-complex on $V(H) - \ell$ vertices,  obtained from the complex $\Hy$ generated by the down-closure of~$H$ by deleting the vertices $x_1,\ldots,x_\ell$ (and all edges containing them) and deleting all the edges of size $k$. Given a $(t_0, t_1, \eps)$-equitable $(k-1)$-complex~$\J$ on~$V(G)$, the \emph{number of rooted copies of~$H$ supported by~$\J$}, written $n_H(G;v_1,\ldots,v_\ell,\J)$, is defined as the number of labelled rooted copies of~$H$ in~$G$ such that every vertex of~$\Hskel$ lies in a distinct cluster of~$\J$ and the image of~$\Hskel$ is in~$\J$ (note that we do \emph{not} require the edges involving
$v_1,\ldots,v_\ell$ to be contained in or supported by $\J$ and
typically they will not be so). We also define $n'_{\Hskel}(\J)$ to be the number of labelled copies of~$\Hskel$ in~$\J$ with each vertex of~$\Hskel$ embedded in a distinct cluster of~$\J$. Then the
\emph{density $d_H(G;v_1,\ldots,v_\ell,\J)$ of rooted copies of~$H$ in~$G$ supported by~$\J$} is defined as
$$d_H(G;v_1,\ldots,v_\ell,\J):= \frac{n_H(G;v_1,\ldots,v_\ell,\J)}{n'_{\Hskel}(\J)}.$$
Again there is a natural probabilistic interpretation: let $\psi : V(\Hskel) \to V(G)$ be an injective map chosen uniformly at random. Extend $\psi$ to a map $\psi' : V(H) \to V(G)$ by taking $\psi'(x_i) = v_i$ for each~$i \in [\ell]$. Then $d_H(G;v_1,\ldots,v_\ell,\J)$ is the conditional probability that~$\psi'$ embeds $H$ in~$G$, given that~$\psi$ embeds $\Hskel$ in~$\J$ with every vertex of~$\Hskel$ embedded in a different cluster of~$\J$.

We now state the statement of the Regular Slice Lemma that we need, that is a straightforward simplification of~\cite[Lemma 10]{ABCM}.

\begin{lemma}[Regular Slice Lemma {\cite[Lemma 10]{ABCM}}]
\label{lem: regular slice}
Let~$k \geq 3$. For all positive integers~$t_0$, positive~$\eps_k$ and all functions~$r: \N \to \N$ and~$\eps: \N \to (0, 1]$, there are integers~$t_1$ and~$n_2$ such that the following holds for all~$n \ge n_2$ which are divisible by~$t_1!$. Let~$G$ be a~$k$-graph whose vertex set~$V$ has size~$n$. Then there exists a~$(k-1)$-complex~$\J$ on~$V$ which is a~$(t_0, t_1, \eps(t_1), \eps_k, r(t_1))$-regular slice for~$G$ such that
\begin{enumerate}[label={\rm(\roman*)}]
    \item  for each~$i \in [k-1]$, each set~$Y$ of~$i$ clusters of~$\J$, we have~$\overline{\deg}(Y; R(G)) = \overline{\deg}(Y; G) \pm \eps_k$;
    \item \label{itm: skeleton} for  each~$1\le\ell\le 1/\eps_k$,  each $k$-graph $H$ equipped with a set of distinct root vertices  $x_1,\ldots,x_\ell$ such that $|V(H)|\le 1/\eps_k$ and any distinct vertices  $v_1,\ldots,v_\ell$ in~$V$, we have
  \[\big|d_H(G;v_1,\ldots,v_\ell,\J)-d_H(G;v_1,\ldots,v_\ell) \big|<\eps_k
  \,.\]
\end{enumerate}
 
\end{lemma}

Given a regular slice~$\J$ for a~$k$-graph~$G$, in order to work with~$k$-tuples of regular and dense clusters, we define the~$d$-reduced~$k$-graph~$\R_d^{\J}(G)$ as follows.

\begin{definition}[The~$d$-reduced~$k$-graph]
Let~$k \geq 3$.
Let~$G$ be a~$k$-graph and suppose $\J$ is a~$(t_0, t_1, \eps, \eps_k, r)$-regular slice for~$G$ with~$t$ clusters where~$t_0 \le t\le t_1$. Then, for~$d >0$, we define the \emph{$d$-reduced~$k$-graph~$\R_d^{\J}(G)$} to be the~$k$-graph whose vertices are the set~$[t]$ (corresponding to the~$t$ clusters of~$\J$) and whose edges are~$k$-sets in~$\binom{[t]}{k}$ such that~$e \in E\left(\R_d^{\J}(G)\right)$ if and only if the corresponding~$k$-set~$X$ of clusters of~$\J$ is such that~$G$ is~$(\eps_k, r)$-regular with respect to~$X$ and~$d^*(X) \geq d$.
\end{definition}

We write $R_d(G)$ for~$R_d^\J(G)$ when $\J$ is clear from context. The next lemma states that the codegree conditions are also preserved by~$R_d(G)$, so we can work with this structure.

\begin{lemma}[{\cite[Lemma 12]{ABCM}}] \label{Lma: degree conditions in d reduced k graph}
    Let~$G$ be a~$k$-graph and let~$\J$ be a~$(\cdot, \cdot, \eps, \eps_k, r)$-regular slice for~$G$ with~$t$ clusters. For any set~$S$ of at most~$k-1$ vertices of~$R_d(G)$, let~$S_R$ be the set of the corresponding at most~$k-1$ clusters of~$\J$. Then we have 

   ~$${\overline{\deg}}(S; R_d(G)) \ge \overline{\deg}(S_R; R(G)) - d - \zeta(S_R),$$ where~$\zeta(S_R)$ is the proportion of~$k$-sets of clusters~$T$ with~$S_R \subseteq T$ which are not~$(\eps_k, r)$-regular with respect to~$G$.
\end{lemma}

Given a copy of some subgraph~$H' \subseteq H$ in~$\cG^{(k)}$, how many ways are there to extend~$H'$ to a copy of~$H$ in~$\cG^{(k)}$? The next lemma gives a lower bound on this number for almost all copies of~$H'$ in~$\cG^{(k)}$. To state this precisely, we define the following. 

Let~$\cG$ be an~$s$-partite~$k$-complex whose vertex classes~$V_1, \dots ,V_s$
are each of size~$m$ and let~$\Hy$ be an~$s$-partite~$k$-complex whose vertex
classes~$X_1, \dots, X_s$ each have size at most~$m$. We say that an embedding of~$\Hy$ in~$\cG$ is
\emph{partition-respecting} if for any~$i \in [s]$ the vertices of~$X_i$ are
embedded within~$V_i$. On the other hand, we say $\cG$ respects the partition of~$\Hy$ if whenever $\cG$ contains an $i$-edge with vertices in~$V_{j_1}, \dots, V_{j_i}$, then $\Hy$ contains an $i$-edge with vertices in~$X_{j_1}, \dots, X_{j_i}$. We denote the set of labelled partition-respecting copies
of~$\Hy$ in~$\cG$ by~$\Hy_\cG$. The Extension Lemma
states that if~$\Hy'$ is an induced subcomplex of~$\Hy$ and~$\cG$ is regular with~$\cG^{(k)}$ reasonably dense, then almost all partition-respecting copies of
$\Hy'$ in~$\cG$ can be extended to a large number of copies of~$\Hy$ in
$\cG$. 

\begin{lemma}[Extension Lemma {\cite[Lemma 29]{ABCM}}]\label{lma :ext}
Let~$k,s,r,b,b',m$ be positive integers, such that $b'<b$ and suppose~$c,
\beta,d_2,\ldots,d_k,\eps,\eps_k$ be positive constants such that~$1/d_i
\in\N$ for any~$2 \leq i \leq k-1$ and 
\[ 1/m\ll 1/r,
\eps \ll c \ll \eps_k,d_2,\ldots,d_{k-1}\quad\text{and}\quad \eps_k \ll
\beta,d_k, 1/s,1/b\,.\]
Suppose that~$\Hy$ is an~$s$-partite~$k$-complex on~$b$ vertices
with vertex classes~$X_1, \dots, X_s$ and let~$\Hy'$ be an induced subcomplex
of~$\Hy$ on~$b'$ vertices.
Suppose that~$\cG$ is an~$s$-partite~$k$-complex with vertex classes~$V_1,\ldots,V_s$, all of size~$m$, such that~$\bigcup_{0 \le i \le k-1} \cG^{(i)}$ is~$(\cdot, \cdot, \eps)$-equitable with density vector
$(d_{k-1},\ldots,d_2)$.  Suppose further that for each~$e\in\Hy^{(k)}$ with index~$A \in \binom{[s]}{k}$, the~$k$-graph~$\cG^{(k)}[V_A]$ is~$(d,\eps_k,r)$-regular with respect to~$\cG^{(k-1)}[V_A]$ for some~$d\ge d_k$. Then
all but at most~$\beta |\Hy'_\cG|$ labelled partition-respecting copies of~$\Hy'$
in~$\cG$ extend to at least~$cm^{b-b'}$ labelled partition-respecting
copies of~$\Hy$ in~$\cG$. \qed
\end{lemma}

We will use the following result due to K\"uhn, Mycroft and Osthus~\cite{Kuhn-Mycroft-Osthus}. The \emph{maximum vertex degree} of a complex $\cG$ is the maximum degree of a vertex of~$\cG$. 

\begin{lemma}[Embedding Lemma~{\cite[Lemma 4.5]{Kuhn-Mycroft-Osthus}}]\label{lma: Embedding Lemma}
Let~$\Delta, k, s, r, m_0$ be positive integers and let $c, \eps, \eps_k$ and $d_2, \dots, d_k$ be positive constants such that~$1/d_i \in \N$ for all~$i < k$,

$$1/m_0 \ll 1/r, \eps \ll \min\{\eps_k, d_2, \dots, d_{k-1}\} \le \eps_k \ll d_k, 1/\Delta, 1/s$$ and~$$c \ll d_2, \dots, d_k.$$
Then the following holds for all integers~$m \ge m_0$. Suppose that~$\cG$ is an~$s$-partite~$k$-complex of maximum vertex degree at most~$\Delta$ with vertex classes~$X_1, \dots, X_s$ such that~$|X_i| \le cm$ for each~$i \in [s]$. Suppose also that~$\Hy$ is a~$(\mathbf{d}, \eps_k, \eps, r)$-regular~$s$-partite~$k$-complex with vertex classes~$V_1, \dots, V_s$ all of size~$m$, respecting the partition of~$\cG$. Then~$\Hy$ contains a labelled partition-respecting copy of~$\cG$.    
\end{lemma}

We now show that if~$H$ is a~$3$-graph with~$\delta_2(H) \ge (\gamma + \alpha)n$, then the reduced graph `almost' inherits the codegree condition. We borrow the following terminology from Han, Lo and Sanhueza-Matamala~\cite{strongdense} to describe this. For $0 \le \mu, \theta \le 1$ and a $k$-graph $H$ on $n$ vertices, $H$ is said to be \emph{$(\mu, \theta)$-dense} if there exists $\mathscr{S} \subseteq \binom{V(H)}{k-1}$ of size at most $\theta \binom{n}{k-1}$ such that for all $S \in \binom{V(H)}{k-1}\setminus\mathscr{S}$ we have $\deg_H(S) \ge \mu(n-k+1)$.

\begin{proposition} \label{Prop: almost codegree in the reduced graph}
    Let~$1/t \ll d_3, \eps_3 \ll \alpha, \beta$ and $0\le \gamma\le 1$. Let~$H$ be a~$3$-graph on~$n$ vertices with~$\delta_2(H) \ge (\gamma+\alpha)n$ and~$\J$ be a~$(\cdot, \cdot, \eps, \eps_3, r)$-regular slice for~$H$ with~$t$ clusters such that for any set $Y$ of~$k-1$ clusters of~$\J$, we have $\overline{\deg}(Y; R(H)) = \overline{\deg}(Y; H) \pm \eps_3$. Let $R = R_{d_3}^{\J}(H)$. Then at least~$(1-\beta)\binom{t}{2}$ pairs of vertices in~$R$ have codegree at least~$(\gamma + \alpha/2)t$, i.e. $R$ is $(\gamma+\alpha/2, \beta)$-dense.
    
\end{proposition}
\begin{proof}

Let~$V_1, V_2 \in V(R(H))$, so~$V_1$ and~$V_2$ are clusters of~$\J$. We have~$$\overline{\deg}(V_1V_2; R(H)) = \overline{\deg}(V_1V_2; H) \pm \eps_3 \ge \gamma+ 3\alpha/4.$$ Since $\J$ is a regular slice,~$H$ is $(\eps_3, r)$-regular with respect to all but at most ~$\eps_3 \binom{t}{3}$ triples of clusters of~$\J$. So at most~$\frac{1}{2}\sqrt{\eps_3}\binom{t}{2}$ pairs of clusters of~$\J$ lie in at least~$2\sqrt{\eps_3}t$ of irregular triples. Call a pair of clusters in~$R(H)$ `good' if they lie in less than~$2\sqrt{\eps_3}t$ irregular triples. The number of good pairs of clusters of~$\J$ is at least~$(1-\frac{1}{2}\sqrt{\eps_3})\binom{t}{2} \ge (1-\beta)\binom{t}{2}$. 

For~$i \in [t]$, let~$i$ be the vertex in~$R$ for which~$V_i$ is the corresponding cluster of~$\J$. Lemma~\ref{Lma: degree conditions in d reduced k graph} implies that for any good pair of clusters~$V_i, V_j \in R(H)$, we have 
$${\deg}(ij; R) \ge \left(\overline{\deg}(V_iV_j; R(H)) - d_3 - \frac{2\sqrt{\eps_3}t}{t-2}\right)(t-2) \ge \left(\gamma + \frac{\alpha}{2}\right) t.$$
This shows that all but at most~$\beta\binom{t}{2}$ pairs of vertices have codegree at least~$(\gamma + \alpha/2)t$ in~$R$.
\end{proof}

Say that a $3$-graph $H$ is \emph{strongly $(\mu, \theta)$-dense} if it is $(\mu, \theta)$-dense and for all edges $e \in E(H)$ and all pairs $X \subseteq e$ we have $\deg_H(X) \ge \mu(n-2)$. We use the following result by Han, Lo and Sanhueza-Matamala~\cite{strongdense}. Their statement of the lemma actually applied to all $k \geq 3$, and did not include the final ``moreover'' statement, but this can be read off from their proof.

\begin{lemma}[{\cite[Lemma 8.4 for~$k=3$]{strongdense}}] \label{lma: find a strongly dense subgraph}
  Let $n \ge 6$ and $0< \mu, \theta <1$. Let $H$ be a $3$-graph on~$n$ vertices that is $(\mu, \theta)$-dense. Then there exists a subgraph $H'$ on $V(H)$ that is strongly $(\mu - 8 \theta^{1/4}, \theta + \theta^{1/4})$-dense. Moreover, $|E(H-H')|\le 8\theta^{1/2}\binom{n}{3}$. 
\end{lemma}

\section{Connecting Lemma} \label{sec: Connecting Lemma}

The main aim of this section is to prove the Connecting Lemma (Lemma~\ref{lma: simpler connecting lemma}). A rough sketch of the proof of the Connecting Lemma is as follows. In the proof of Lemma~\ref{lma: simpler connecting lemma}, we use the strong Hypergraph Regularity, particularly, the Regular Slice Lemma (Lemma~\ref{lem: regular slice}) proved by Allen, B\"ottcher, Cooley and Mycroft~\cite{ABCM}. As a consequence of the Regular Slice Lemma, an appropriate reduced~$3$-graph of~$H$ inherits the minimum codegree condition approximately. In order to convert this to a minimum codegree condition, we use a result (Lemma~\ref{Lma: min codeg}) by Ferber and Kwan~\cite{Ferber-Kwan}. Once we have a minimum codegree in the reduced~$3$-graph, we use Lemma~\ref{lma: tight connectivity at 7n/9} to get a squared-tight-walk between any two triples of vertices of the reduced~$3$-graph. Note that a tight walk lets us fix the order of the initial triple but not the final one. In order to have a squared-tight-walk that respects the order of both the initial and final triples of vertices, we preserve a~$K_5^3$ in the reduced graph. We then use the Extension Lemma (Lemma~\ref{lma :ext}) due to Allen, B\"ottcher, Cooley and Mycroft to get a squared-tight-path between almost all pairs of ordered vertex-triples supported on the regular slice (see Section~\ref{sec: Hypergraph Regularity} for the definition of a regular slice) on a constant number of vertices. Each pair of vertex-triples~$(\mathbf{e}_1, \mathbf{e}_2)$ in~$H$ locally extends to one such pair of ordered vertex-triples supported on the regular slice, avoiding the small forbidden set~$X$.

The next result by Ferber and Kwan~\cite{Ferber-Kwan}
provides a way of translating an `almost minimum codegree' condition to an exact minimum codegree condition on a suitable subgraph.

\begin{lemma}[{\cite[Lemma 3.4]{Ferber-Kwan}}] \label{Lma: min codeg} 
Let~$G$ be a $k$-graph on $n$ vertices in which at most~$\delta \binom{n}{\ell}$ sets $X \in \binom{V(G)}{\ell}$ have $\deg_G(X) < (\mu + \eta)\binom{n-\ell}{k-\ell}$. Also let $Q \geq 2\ell$, and choose a set~$S \subseteq V(G)$ of size $|S| = Q$ uniformly at random. Then with probability at least~$1 - \binom{Q}{\ell}(\delta + e^{-\eta^2 Q/4k^2})$ the induced subgraph~$G[S]$ has minimum $\ell$-degree $\delta_{\ell}(G[S]) \geq (\mu + \eta/2)\binom{Q-\ell}{k-\ell}$.
\end{lemma}

We first find a short squared-tight-walk between any two ordered triples of vertices that are edges of an appropriate reduced graph, respecting the order of both triples. 

\begin{lemma} \label{lma: find a walk in the reduced graph}
    Let $1/t \ll \eta, \beta \ll \alpha$. Let $R$ be a $3$-graph on $t$ vertices that is strongly $(7/9 + \alpha, \beta)$-dense, and let $\mathbf{x}$ and $\mathbf{y}$ be ordered triples which are edges of~$R$. Then there exists a squared-tight-walk on at most~$1/\eta$ vertices from~$\mathbf{x}$ to~$\mathbf{y}$.
\end{lemma}
\begin{proof}
Fix~$T$ such that $1/t \ll \eta, \beta \ll 1/T \ll \alpha$. Let $W \subseteq V(R)$ be a set of~$T$ vertices chosen uniformly at random. By Lemma~\ref{Lma: min codeg} with $R, 3, t, T, \beta, 7/9, \alpha, 2$ playing the roles of $G, k, n, Q, \delta, \mu, \eta, \ell$, respectively, 

$$\mathbb{P}\left( \delta_2(R[W]) \ge \left(\frac{7}{9} + \frac{\alpha}{2} \right) T \right) \ge 1 - \binom{T}{2}\left(\beta + e^{-\alpha^2 T / 36} \right) \ge \frac{99}{100}.$$

\noindent Note that for all $x, y \in V(R)$ with $\deg_R(xy)>0$, $\mathbb{E}
\deg_R(xy, W) = (7/9 + \alpha) |W|$ and $\deg_R(xy, W) \sim \mathrm{Hyp}(t, T, \deg_R(xy))$. By Hoeffding's inequality (Lemma~\ref{Lma: Chernoff hyp}), we have $$\mathbb{P}\left( \deg_R(xy, W) \le (7/9 + \alpha/2)T \right) \le 2e^{-7\alpha^2T/108}.$$ Taking a union bound over all elements of $\{xy \in \binom{V(H)}{2} \colon \deg_R(xy)>0\}$, we deduce that with high probability, we have $\delta_2(R[W]) \ge (7/9 + \alpha/2)T$ and for all $xy \in \binom{V(R)}{2}$ such that $\deg_R(xy)>0$, $\deg_R(xy, W) \ge (7/9 + \alpha/2)T$. Fix such a set~$W$ of vertices. Note that by Theorem~\ref{Thm: codegree density of K35},~$W$ contains a $K_5^3$. Let $V\left(K_5^3\right) = \{v_i \colon  i \in [5]\}$.

We now connect $\mathbf{x}$ and $\mathbf{y}$ via~$W$. Let $\mathbf{x} = x_1x_2x_3, \mathbf{y}=y_1y_2y_3$ and $x_4, x_5, x_6 \in W \setminus\{x_i, y_i \colon i \in [3]\}$ be distinct vertices such that $x_1x_2x_3x_4x_5x_6$ is a squared-tight-path. There are at least $(T/3)^3$ ways to pick $x_4, x_5, x_5 \in W \setminus\{x_i, y_i \colon i \in [3]\}$. Similarly, fix $y_4, y_5, y_6 \in W\setminus\{x_i, x_{i+3}, y_i \colon i \in [3]\}$ such that $y_3y_2y_1y_4y_5y_6$ is a squared-tight-path. By Lemma~\ref{lma: tight connectivity at 7n/9}, $\T(R[W])$ is tightly connected. Thus, there exists a squared-tight-walk $W_1$ from~$x_4x_5x_6$ to~$v_{\sigma(1)}v_{\sigma(2)}v_{\sigma(3)}$ for some $\sigma \in S_3$.  Similarly, there exists a squared-tight-walk $W_2$ from~$y_4y_5y_6$ to~$v_{\sigma'(1)}v_{\sigma'(2)}v_{\sigma'(3)}$ for some $\sigma' \in S_3$.  Let $W'$ be a squared-tight-walk on $\{v_i \colon i \in [5]\}$ from~$v_{\sigma(1)}v_{\sigma(2)}v_{\sigma(3)}$ to~$v_{\sigma'(3)}v_{\sigma'(2)}v_{\sigma'(1)}$. 
Let $\overleftarrow{W_2}$ be the squared-tight-walk obtained from~$W_2$ by reversing its vertex sequence.
Let $W = W_1  W'  \overleftarrow{W_2}$ be the squared-tight-walk from~$x_4x_5x_6$ to~$y_6y_5y_4$.

Note that by deleting all vertices in~$W$ between repeated occurrences of an ordered triple of vertices, we may assume that $W$ contains at most $T^3$ vertices. Thus, $x_1x_2x_3  W  y_1y_2y_3$ is a squared-tight-walk from~$\mathbf{x}$ to~$\mathbf{y}$ on at most $T^3 + 6 \le 1/\eta$ vertices. \end{proof}

We now find a short squared-tight-path between any two ordered triples of vertices that are edges of~$H$, respecting the order of both triples.  

\begin{lemma} \label{lma: Connecting Lemma}
    Let $1/n \ll \eta \ll \alpha$. Let $H$ be a $3$-graph on $n$ vertices with $\delta_2(H) \ge (7/9 + \alpha)n$, and let $\mathbf{x}$ and $\mathbf{y}$ be ordered triples which are edges of~$H$. Then there is a squared-tight-path in~$H$ on at most~$1/\eta$ vertices from~$\mathbf{x}$ to~$\mathbf{y}$.
\end{lemma}
\begin{proof}
Choose constants: $$1/n  \ll 1/t_1 \ll 1/t_0  \ll 1/r, \eps\ll c\ll \eps_3 \ll \beta' \ll d_2, d_3, \eta \ll \theta, \beta \ll \alpha.$$
Let $\eta' = \left(\lfloor 1/\eta \rfloor-6\right)^{-1}$, $\mathbf{x} = x_1x_2x_3$ and $\mathbf{y}=y_1y_2y_3$. Let $H'$ be an induced subgraph of~$H$ on $n'$ vertices by deleting at most $t_1!$ arbitrary vertices from~$V(H)\setminus\{x_i, y_i \colon i \in [3]\}$, such that $n' \equiv 0 \pmod{t_1!}$. Note that $\delta_2(H') \ge (7/9 + \alpha)n - t_1! \ge (7/9+5\alpha/6)n'$. By Lemma~\ref{lem: regular slice} with $H', n', 3$ playing the roles of~$G, n, k$, respectively, let~$\J$ be the $(\cdot, \cdot, \eps, \eps_3, r)$-regular slice for~$H'$ on~$t_0 \le t \le t_1$ clusters $V_1, \dots, V_t$, density parameter $d_2$ and let $R = R_{d_3}(H')$. By Proposition~\ref{Prop: almost codegree in the reduced graph} with $H', n', 5\alpha/6, \theta, 7/9$ playing the roles of $H, n, \alpha, \beta, \gamma$, respectively,~$R$ is $(7/9 + 5\alpha/12, \theta)$-dense.  By Lemma~\ref{lma: find a strongly dense subgraph}, with $R$ and $7/9 + 5\alpha/12$ playing the roles of~$H$ and $\mu$ respectively, there exists a subgraph $R'$ of~$R$ which is strongly $(7/9 + \alpha/3, 2\theta^{1/4})$-dense. Moreover, $|E(R - R')|\le 8\sqrt{\theta}\binom{t}{3}$.

 We say that the ordered triple $x_4x_5x_6$ of vertices is an \emph{extension} of the ordered triple $x_1x_2x_3$ if $x_1x_2x_3x_4x_5x_6$ is a squared-tight-path in~$H'$.  
 for~$ijk \in E(R)$, let  $\J_{ijk} = \{ v_iv_jv_k \in V_i\times V_j \times V_k \colon \{v_i,v_j,v_k\} \in K_3(\J) \}$.
 Note that $ijk$ and $v_iv_jv_k$ are considered as ordered triples.
    \begin{claim} \label{clm: subslice}
       Let $I \subseteq [t]$ be such that $|I| \le 3$. Then there exists an ordered triple $ijk \in \binom{V(R') \setminus I}{3} $ and $\J_{ijk}' \subseteq \J_{ijk}$ such that $ijk \in E(R')$, $|\J_{ijk}'|\ge |\J_{ijk}|/27$ and each member of~$\J_{ijk}'$ is an extension of~$x_1x_2x_3$. 
    \end{claim}
\begin{proofclaim}
Let $\tilde{H}$ be a squared-tight-path on $6$ vertices from~$x_1x_2x_3$. 
Observe that by the codegree of~$H'$ and Fact~\ref{fct: tetrahedral codegree}, the number of ordered triples $x_4x_5x_6$ such that $x_1\dots x_6$ forms an extension of~$x_1x_2x_3$ in~$H'$ is at least $$\left(\frac{1}{3} + \frac{5\alpha}{2}\right)n'\left(\left(\frac{1}{3} + \frac{5\alpha}{2}\right)n'-1\right)\left(\left(\frac{1}{3} + \frac{5\alpha}{2}\right)n'-2\right) \ge \left(\frac{1}{27} + \alpha\right)n'^3.$$
By Lemma~\ref{lem: regular slice}\ref{itm: skeleton} with $\tilde{H}, 3, \eps_3$ playing the role of~$H, \ell, \eps_k$, respectively, we have $$d_{\tilde{H}} (H'; x_1, x_2, x_3, \J) \ge d_{\tilde{H}}(H', x_1, x_2, x_3) - \eps_3 \ge 1/27 + \alpha - \eps_3 \ge 1/27 + \alpha/2.$$

\noindent By deleting the proportion of triples of clusters which do not correspond to edges of~$R'$ and those involving $I$, we are left with at least \begin{align*} 
    1/27 + \alpha/2 - d_3 - \eps_3 - 8\sqrt{\theta} - 3/t \ge 1/27 + \alpha/4
\end{align*} proportion of the triples of vertices. Each of these triples of vertices is supported on triples of clusters which correspond to edges in~$R' \setminus I$. By averaging over all edges in~$R'\setminus I$, there exists $ijk \in E(R')\setminus I$ such that \begin{align*} \label{eqn: proportion retained on J_ijk}
    d_{\tilde{H}}(H'; x_1, x_2, x_3, \J_{ijk}) \ge 1/27 + \alpha/4.
\end{align*} Let $\J_{ijk}' = \{v_iv_jv_k \in \J_{ijk} \colon v_iv_jv_k \text{ is an extension of  } x_1x_2x_3\}.$ Therefore, we have $|\J_{ijk}'| \ge |\J_{ijk}|/27$.   \end{proofclaim}    

Let $ijk$ and $\J_{ijk}'$ be given by Claim~\ref{clm: subslice} with $I = \emptyset$. Applying Claim~\ref{clm: subslice} again with $I = \{i,j,k\}$ and $x_1x_2x_3 = y_3y_2y_1$ (as an ordered triple), there exists an ordered triple $i'j'k' \in E(R')$ with $ijk \cap i'j'k' = \emptyset$ and $\J_{i'j'k'}' \subseteq \J_{i'j'k'}$ such that for all ordered triples $v_{i'}v_{j'}v_{k'} \in \J_{i'j'k'}'$, we have $v_{i'}v_{j'}v_{k'}$ is an extension of~$y_3y_2y_1$ and $|\J_{i'j'k'}'|\ge |\J_{i'j'k'}|/27$. 

 By Lemma~\ref{lma: find a walk in the reduced graph} with $R', \alpha/3, 2\theta^{1/4}, \eta', ijk, k'j'i'$ playing the roles of~$R, \alpha, \beta, \eta, \mathbf{x}, \mathbf{y}$, respectively, there exists a squared-tight-walk $W$ in~$R'$ of length at most $1/\eta'$ from~$ijk$ to~$k'j'i'$. Let $W = w_1w_2\dots w_b$ where $b \le 1/\eta'$, $w_1w_2w_3 = ijk$, $w_{b-2}w_{b-1}w_b = k'j'i'$. Without loss of generality, we may assume the distinct vertices appearing in~$W$ are $[s]$ and when counted with multiplicity, there are~$b$ vertices. Note that $W$ is $s$-partite and $6 \le s \le b \le 1/\eta'$.  

Let~$\Hy^{(3)}$ be the squared-tight-path formed by replacing each repeated vertex of~$W$ with distinct copies of it, whenever it is repeated in~$W$. Note that~$\Hy^{(3)}$ is~$s$-partite with vertex classes~$V_1, \dots, V_s$ where~$\{V_i\}_{i \in [b]}$ is the set of copies of the vertex classes of~$W$. Let~$\Hy$ be the down-closure of~$\Hy^{(3)}$, so~$\Hy$ is an~$s$-partite~$3$-complex. Let~$\Hy'$ be the subcomplex induced on the first and final three vertices of the squared-tight-path, i.e. on the vertex classes corresponding to~$ijk$ and~$k'j'i'$.

Let~$\cG$ be the~$s$-partite~$3$-complex obtained from~$\J\left[\bigcup_{p \in [s]} V_p \right]$ by adding the edges of~$H'$ supported on $\J_{ijk}'$ and $\J_{i'j'k'}'$ for the triples of vertices $w_1w_2w_3 = ijk$ and $w_{b-2}w_{b-1}w_b = k'j'i'$, the edges supported on $\J\left[\bigcup_{p \in [b-2]}V_{w_p} V_{w_{p+1}} V_{w_{p+2}} \right]$ for all vertex triples $\{w_p w_{p+1} w_{p+2} \colon p \in [b-3]\setminus\{1\}\}$ and all the edges of~$H'$ supported on~$\J$ for any other triple. Note that for any edge~$e \in \Hy^{(3)}$ with index~$A \in \binom{[s]}{3}$, the~$3$-graph~$\cG^{(3)}[V_A]$ is either~$(d, \eps_3, r)$-regular with~$d \ge d_3$ (on the triples of clusters that are edges of~$W$) or~$(1, \eps_3, r)$-regular (on any other triple) with respect to~$\cG^{(2)}[V_A]$.

All the conditions to apply Lemma~\ref{lma :ext} are thus satisfied. Consider a pair of ordered triples of vertices,~$v_1v_2v_3$ from~$V_{w_1}V_{w_2}V_{w_3}$ and~$v_{b-2}v_{b-1}v_b$ from~$V_{w_{b-2}} V_{w_{b-1}} V_{w_b}$. Note that by our assumption on~$ijk$ and~$i'j'k'$ to be vertex-disjoint, so are~$v_1v_2v_3$ and~$v_{b-2}v_{b-1}v_b$. By Lemma~\ref{lma :ext} with~$3, 6, \beta'$ playing the roles of~$k, b', \beta$, respectively, for all but at most~$\beta'$ proportion of all ordered pairs of triples of vertices $(v_1v_2v_3, v_{b-2}v_{b-1}v_b)$ from the clusters of~$\Hy'$, there are at least~$cm^{b-6} \ge cm^3$ extensions to labelled partition-respecting copies of~$\Hy$ in~$\cG$. Each such copy of~$\Hy$ corresponds to a squared-tight-path in~$H'$ on~$b$ vertices with every vertex in the required cluster. Since $\beta' \ll d_2$, $|\J_{ijk}'| \ge |\J_{ijk}|/27 \text{ and } |\J_{i'j'k'}'| \ge |\J_{i'j'k'}'|/27$ by Claim~\ref{clm: subslice}, at least one such copy of~$\Hy^{(3)}$ is a squared-tight-path going from an extension of~$x_1x_2x_3$ to an extension of~$y_3y_2y_1$. Let $P$ be one such squared-tight-path. Then $x_1x_2x_3 P  y_1y_2y_3$ is the required squared-tight-path in~$H$ from~$\mathbf{x}$ to~$\mathbf{y}$ on at most $b+6 \le 1/\eta' + 6 \leq 1/\eta$ vertices. \end{proof}

We are finally ready to prove Lemma~\ref{lma: simpler connecting lemma}.

\begin{proof}[Proof of Lemma~\ref{lma: simpler connecting lemma}]
Let $\mathcal{E} := \{\mathbf{x}_i, \mathbf{y}_i \colon i \in [s]\}$, and for each~$i \in [s]$ let $X_i := (X \cup V(\mathscr{E})) \setminus (V(\mathbf{x}_i) \cup V(\mathbf{y}_i))$. Let $\ell \in [s]$ be maximal with the property that there exists a collection $\mathscr{P} = \{P_1, \dots, P_{\ell}\}$ of vertex-disjoint squared-tight-paths such that for each~$i \in [\ell]$ the path $P_i$ has initial triple $\mathbf{x}_i$, final triple $\mathbf{y}_i$, order $|V(P_i)| \le \psi^{-1/2} $ and  $V(P_i) \cap X_i = \emptyset$. 

Suppose for a contradiction that $\ell < s$. Let $X' := X_{\ell+1} \cup \bigcup_{i \in [\ell]}V(P_i)$, so $|X'| \leq |X| + 6s + \psi^{-1/2} \ell \leq \psi n + 6 \psi n + \sqrt{\psi} n \leq 8 \sqrt{\psi}n$. Also let $H' := H \setminus X'$ and $n' := |V(H')|$, so $\delta_2(H') \ge (7/9 + \alpha)n - 8 \sqrt{\psi} n \ge (7/9 + \alpha/2)n'$.
Lemma~\ref{lma: Connecting Lemma}, with $H', n', \alpha/2, \mathbf{x}_{\ell+1}, \mathbf{y}_{\ell+1}$ and $\sqrt{\psi}$ playing the roles of $H, n, \alpha, \mathbf{x}, \mathbf{y}$ and~$\eta$ respectively, implies that there exists a squared-tight-path $P_{\ell+1}$ in~$H'$ from~$\mathbf{x}_{\ell+1}$ to~$\mathbf{y}_{\ell+1}$ on at most~$\psi^{-1/2}$ vertices. The collection $\mathscr{P} \cup \{P_{\ell+1}\}$ then contradicts the maximality of~$\ell$. We conclude that $\ell = s$, and the collection $\mathscr{P}$ is as desired, since  $|V(\mathscr{P})| \le \psi^{-1/2} s \le \sqrt{\psi} n$.\end{proof}

\section{Absorption} \label{sec: Absorption}
The aim of this section is to prove Lemma~\ref{lma: absorption lemma}. Our approach loosely follows the methods used by Ara\'ujo, Piga and Schacht~\cite{PigaAraujoSchacht} and H\`an, Person and Schacht~\cite{han-person-schacht}, who in turn followed the work of Rödl, Ruciński and Szemerédi (see, e.g.~\cite{RRS3graph, RRSkgraph}).

We first formally define an absorber for our problem. Let $H$ be a $3$-graph on vertex set~$V$ and let $(v_1, v_2, v_3, v_4) \in V^{4}$. A labelled set of~$36$ vertices $A = \{x_i, y_i, z_i, u_{i, j} \colon  i \in [4], \text{ }  j \in [6]\}$ is an \emph{absorber for} $\mathbf{v} := (v_1, v_2, v_3, v_4)$ if the ordered sequences \begin{enumerate}[label={\rm(\roman*)}]
    \item $T_1(A):= x_1\dots x_4 y_1 \dots y_4 z_1 \dots z_4$, \label{itm: K4 absorption}
    \item $T_2(A):= x_1 \dots x_4 z_1 \dots z_4$, \label{K4 no absorption}
    \item $U_{i}(A, \mathbf{v}):= u_{i, 1}u_{i, 2} u_{i, 3} v_i u_{i, 4} u_{i, 5} u_{i, 6}$ \label{6 vertices absorption} for each~$i \in [4]$ and
    \item $U_{i}(A):= u_{i, 1}u_{i, 2} u_{i, 3} y_i u_{i, 4} u_{i, 5} u_{i, 6}$ \label{6 vertices no absorption} for each~$i \in [4]$ 
\end{enumerate} are each squared-tight-paths in~$H$. We say that $A$ is an \emph{absorber} if it is an absorber for some $(v_1, v_2, v_3, v_4) \in V^4$. We also write $f_{A, \mathbf{v}}$ for the function which maps $T_2(A)$ to~$T_1(A)$ and, for each~$i \in [4]$, maps $U_{i}(A)$ to~$U_{i}(A, \mathbf{v})$.

The purpose of this definition is the following. Suppose $A$ is an absorber for~$\mathbf{v} := (v_1, v_2, v_3, v_4)$, where $v_1, v_2, v_3$ and $v_4$ are distinct vertices of~$H$. Then 
$$\mathcal{P}_A := \{T_2(A), U_{1}(A), U_{2}(A), U_{3}(A), U_{4}(A)\}$$ is a collection of five vertex-disjoint squared-tight-paths with vertex set $V(\mathcal{P}_A) = A$.  Moreover, 
$$\mathcal{Q}_{A, \mathbf{v}} := \{f_{A, \mathbf{v}}(P) \colon P \in \mathcal{P}_A \} = \{T_1(A), U_{1}(A, \mathbf{v}), U_{2}(A, \mathbf{v}), U_{3}(A, \mathbf{v}), U_{4}(A, \mathbf{v})\}$$ 
is a collection of five vertex-disjoint squared-tight-paths with vertex set $V(\mathcal{Q}_{A, \mathbf{v}}) = A \cup \{v_1, v_2, v_3, v_4\}$, and each path in~$\mathcal{Q}_{A, \mathbf{v}}$ has the same ends as the corresponding path in~$\mathcal{P}_A$. So the effect of replacing each path $P \in \mathcal{P}_A$ by $f_{A, \mathbf{v}}(P)$ is to `absorb' the vertices $v_1, v_2, v_3$ and~$v_4$ into the squared-tight-paths of the absorber.

We first show that, for each~$(v_1, v_2, v_3, v_4) \in V^4$, a constant proportion of all labelled sets of 36 vertices of~$H$ are absorbers for~$(v_1, v_2, v_3, v_4)$. 
Recall that, for a $3$-graph~$F$, $\gamma(F)$ is the codegree Tur\'an density of~$F$. 

\begin{lemma}\label{lma: many absorbers for each tuple}
Let $1/n \ll c \ll \alpha$, let $H$ be a $3$-graph on $n$ vertices with $\delta_2(H) \ge (7/9+\alpha)n$, and let $v_1$, $v_2$, $v_3$ and $v_4$ be vertices of~$H$. Then there are at least $cn^{36}$ absorbers for~$(v_1, v_2,v_3, v_4)$ in~$H$.  
\end{lemma}

\begin{proof}
Let $\beta=(2c)^{1/5}$.
    By a greedy argument analogous to Fact~\ref{fct: tetrahedral codegree}, we have $\gamma(K_4^3) \le 2/3$. Since $\delta_2(H\setminus \{v_1, v_2, v_3, v_4\}) > 7(n-4)/9$, by Theorem~\ref{lma: supersaturation, blowup density}, there are at least $\beta n^{12}$ copies of~$K_4^3 (3)$ in~$H$ avoiding $v_1, v_2, v_3, v_4$. 
    
    Fix a copy of a $K_4^3(3)$ whose vertex classes are $V_i = \{x_i, y_i, z_i\}$ for each~$i \in [4]$. Note that $x_1\dots x_4 y_1 \dots y_4 z_1 \dots z_4$ and $x_1 \dots x_4 z_1 \dots z_4$ are both squared-tight-paths in~$H$. 

    Fix $i \in [4]$ and consider $L_i = L(v_i) \cap L(y_i) \setminus \{v_1, v_2, v_3, v_4\}$. Note that~$L_i$ is a $2$-graph on $n - 5$ vertices with \begin{equation} \label{eqn: mindegree of L}
        \delta(L_i) \ge 2(7/9 + \alpha)n - n -3\ge (5/9 + \alpha)n.
    \end{equation} This implies that for any edge $a b \in E(L_i)$, we have $$|N_{L_i}(a) \cap N_{L_i}(b)| \ge 2(5/9 + \alpha)n - n \ge (1/9 + 2\alpha)n.$$ The number of triangles in~$L_i$ is therefore at least \begin{equation} \label{eqn: count of triangles in L_i}
        \frac{1}{3} |E(L_i)|\left(\frac{1}{9} + 2\alpha\right)n \ge \frac{5}{81}\binom{n}{3}.
    \end{equation}   

    We say an unordered triple of vertices $w_1, w_2, w_3$ in~$L_i$ is a \emph{nice triple} if $w_1 w_2 w_3$ is a triangle in~$L_i$ as well as $w_1w_2w_3 \in E(H)$. We now show that there is a set of nice triples corresponding to every triangle in~$L_i$. Let $a_1a_2a_3$ be a triangle in~$L_i$. Observe that $$ n \ge \left|\bigcup_{j \in [3]} N_{L_i} (a_j)\right| \ge \sum_{j \in [3]} |N_{L_i}(a_j)| - \sum_{jk \in\binom{[3]}2} |N_{L_i}(a_j) \cap N_{L_i}(a_k)|,$$ implying that $$\sum_{jk \in\binom{[3]}2} |N_{L_i}(a_j) \cap N_{L_i}(a_k)| \stackrel{\eqref{eqn: mindegree of L}}{\ge} 3(5/9 + \alpha)n - n \ge (2/3 + 3\alpha)n.$$ By averaging, we may assume without loss of generality, that $|N_{L_i}(a_1) \cap N_{L_i}(a_2)| \ge (2/9 + \alpha)n$. Using the codegree of~$H$ we deduce that 
    $$|N_H(a_1a_2) \cap N_{L_i}(a_1) \cap N_{L_i}(a_2)| \ge (2/9+\alpha)n - (2/9 - \alpha)n = 2\alpha n.$$
    Therefore, there is a set $B \subseteq V(L_i)$ of at least $2\alpha n$ vertices such that for all $b \in B$, we have $a_1a_2b$ is a triangle in~$L_i$ and $a_1a_2b \in E(H)$. For each triangle in~$L_i$, we thus obtain at least $2\alpha n$ nice triples. Observe that each nice triple can be obtained from at most $3n$ triangles in~$L_i$. So by~\eqref{eqn: count of triangles in L_i} there are at least $$\frac{1}{3n}\left( \frac{5}{81}\binom{n}{3}\right)2\alpha n \ge \frac{\alpha  n^3}{200}$$ nice triples of vertices in~$L_i$. 

    Let $H'$ be an auxiliary $3$-graph on vertex set $V(H)$ whose edges are nice triples of~$L_i$. Since there is a positive density of edges in~$H'$ and the Tur\'an density of an edge is $0$, by Theorem~\ref{lma: supersaturation, blowup density}, there are at least $\beta n^{6}$ $2$-blowups of edges of~$H'$. Observe that a $2$-blowup of an edge in~$H'$ is a set of~$6$ vertices $u_{i, 1}, u_{i, 2}, \dots, u_{i, 6}$ such that $u_{i, 1}\dots u_{i, 6}$ is a squared path in~$L(v_i) \cap L(y_i)$ and furthermore $u_{i, j} u_{i, j+1} u_{i, j+2} \in E(H)$ for all $j \in [4]$. 
    
    By repeating this argument for each~$i \in [4]$, we obtain at least $\beta^5 n^{36}$ choices for an absorber~$A$ for~$(v_1, v_2, v_3, v_4)$. The only way this could fail to be an absorber is if the~$36$ vertices obtained are not all distinct. Removing at most $36^2 n^{35}$ many such sets where the vertices are not distinct, we obtain at least $\beta^5 n^{36} /2 = c n^{36}$ absorbers for~$(v_1, v_2, v_3, v_4)$. This proves the result. \end{proof}

We say two absorbers~$A_1$ and~$A_2$ are \emph{disjoint} if the underlying (unlabelled) sets are disjoint, and otherwise that $A_1$ and $A_2$ \emph{intersect}. In the next lemma we use a random selection to obtain a linear-size family $\F$ of pairwise-disjoint absorbers with the property that every $4$-tuple of vertices of~$H$ has many absorbers in~$\F$. 

\begin{lemma}\label{lma: disjoint absorbers}
 Let~$1/n \ll \beta \ll \alpha$ and let~$H$ be a~$3$-graph on~$n$ vertices with~$\delta_2(H) \ge (7/9 +\alpha)n$. Then there exists a family~$\F$ of at most~$\beta n$ pairwise-disjoint absorbers such that for each~$(v_1, v_2, v_3, v_4) \in V(H)^4$ there are at least~$\beta^2 n$ absorbers for~$(v_1, v_2, v_3, v_4)$ in~$\F$.  
\end{lemma}

\begin{proof}
Let $1/n \ll \beta \ll c \ll \alpha$.
    For each $4$-tuple of vertices~$\textbf{v} = (v_1, \dots. v_4) \in V(H)^4$, let~$\mathcal{A}(\textbf{v})$ denote the set of absorbers for~$\textbf{v}$, so $|\mathcal{A}(\textbf{v})| \ge cn^{36}$ by Lemma~\ref{lma: many absorbers for each tuple}. Choose a family~$\F'$ of labelled sets of size $36$ by selecting each of the~$\binom{n}{36} 36!$ labelled sets with probability~$p = \beta/(2n^{35})$, independently of all other choices. Then $|\F'|$ is a binomial random variable with expectation $$\mathbb{E}|\F'| = p \binom{n}{36} 36! \le \frac{\beta n}{2},$$ 
    so by Chernoff's inequality (Lemma~\ref{Lma: Chernoff}) we have $|\F'| \le \beta n$ with probability at least $3/4$. Similarly, for each~$\textbf{v} \in V(H)^4$ the quantity $|\mathcal{A}(\mathbf{v})\cap \F'|$ is a binomial random variable with $$\mathbb{E}|\mathcal{A}(\textbf{v}) \cap \F'| \ge c n^{36} p =\frac{ c \beta n }{2} .$$ 
    So by applying Chernoff's inequality (Lemma~\ref{Lma: Chernoff}) and taking a union bound over all $\textbf{v} \in V(H)^4$, with probability at least 3/4 it holds that for every $\textbf{v} \in V(H)^4$ we have $|\mathcal{A}(\textbf{v}) \cap \F'| \ge c \beta n /4$.
    Finally, the expected number of pairs of absorbers that intersect is at most~$n^{36} 36^2 n^{35} p^2  = 36^2\beta^2 n/ 4.$
So by Markov's inequality (Lemma~\ref{lma: Markov}), with probability at least~$3/4$ it holds that   
\begin{equation*}
\F' \text{ contains at most } 36^2\beta^2 n \text{ pairs of intersecting absorbers.} 
\end{equation*}

Fix a choice of~$\F'$ for which each of these three events occurs, and let~$\F$ be the subfamily of~$\F'$ formed by deleting from~$\F'$ both members of each intersecting pair of absorbers and also each~$F \in \F'$ which is not an absorber. 
Then $\F$ is a family of at most $\beta n$ pairwise-disjoint absorbers with the property that~$|\mathcal{A}(\textbf{v}) \cap \F| \ge  c \beta n /4 - 2 (36\beta)^2 n \ge \beta^2 n \text{ for each }\textbf{v} \in V(H)^4.$
\end{proof}

By choosing a family of absorbers as in Lemma~\ref{lma: disjoint absorbers} we obtain a collection of pairwise vertex-disjoint squared-tight-paths which can absorb any small set of vertices whose size is a multiple of four (this can be seen by following the proof of the next lemma, ignoring the path~$\hat{P}$ and taking $L' = L$). To be able to absorb sets whose size is not a multiple of four, we also use a copy of~$K^3_5(4)$ in~$H$, which may `donate' up  to three vertices to the set we want to absorb. To be more specific we introduce the following notation: for a vertex sequence $P = v_1 v_2 \dots v_{\ell}$ and a subset $S \subseteq V(P)$, we write $P\setminus S$ to denote the sequence of vertices of~$V(P) \setminus S$ in the same order as in~$P$. Observe that if a copy of~$K^3_5(4)$ in~$H$ has vertex classes $\{a_i, b_i, c_i, d_i\}$ for each~$i \in [4]$, then $P := a_1\dots a_5 b_1 \dots b_5 c_1 \dots c_5 d_1 \dots d_5$ is a squared-tight-path in~$H$ with the property that, for any subset $S \subseteq \{a_5, b_5, c_5\}$, $P \setminus S$ is a squared-tight-path in~$H$ with the same ends as $P$.

\begin{lemma} \label{lma: disconnected absorption lemma}
    Let~$1/n \ll \beta  \ll \alpha$. Let~$H$ be a~$3$-graph on~$n$ vertices with~$\delta_2(H) \ge (7/9 + \alpha)n$. Then there exists a collection $\mathcal{P}$ of pairwise vertex-disjoint squared-tight-paths in~$H$ with $|\mathcal{P}| \leq 5\beta n+1$, with $|V(\mathcal{P})| \le 40\beta n$, and with the following absorbing property. For all sets $L \subseteq V(H) \setminus V(\mathcal{P})$ with $|L| \le \beta^2 n$, there exists a collection $\mathcal{Q}$ of pairwise vertex-disjoint squared-tight-paths in~$H$ with $V(\mathcal{Q}) = V(\mathcal{P}) \cup L$ and a bijection $f: \mathcal{P} \to \mathcal{Q}$ such that for each~$P \in \mathcal{P}$ the paths $P$ and $f(P)$ have the same ends.\end{lemma}

\begin{proof}
Apply Lemma~\ref{lma: disjoint absorbers} to obtain a a family~$\F$ of at most~$\beta n$ pairwise-disjoint absorbers such that for each~$\mathbf{v} \in V(H)^4$ there are at least~$\beta^2 n$ absorbers for~$\mathbf{v}$ in~$\F$, and so that the set $V(\F)$ of vertices covered by $\F$ has $|V(\F)| \leq 36\beta n$.
The latter bound implies that $\delta_2(H \setminus  V(\F)) \ge (7/9 + \alpha/2)n$, so $H \setminus V(\F)$ has at least $$\frac{1}{3} \cdot \binom{n-40\beta n}{2} \cdot \left(\frac{7}{9} + \frac{\alpha}{2}\right)n \geq \frac{7}{9}\binom{n}{3}$$ edges. By Theorem~\ref{Thm: codegree density of K35}, we have $\gamma(K_5^3(4)) = \gamma(K_5^3) \le 0.74$. So by Theorem~\ref{lma: supersaturation, blowup density} there must exist a copy of~$K_5^3 (4)$ in~$H \setminus \F$, say with vertex classes $\{a_k b_k c_k d_k \}$ for each~${k\in [5]}$. Set 
\begin{align*}
    \hat{P} := a_1\dots a_5 b_1 \dots b_5 c_1 \dots c_5 d_1 \dots d_5
\qquad \text{and} \qquad
\mathcal{P} := \{\hat{P}\} \cup \{\mathcal{P}_A \colon A \in \F\}. 
\end{align*}
 
Note that $|\mathcal{P}| = 5|\F| + 1 \leq 5 \beta n +1$ and $|V(\mathcal{P})| \le |V(\F)| + \left|V\left(K_5^3(4)\right)\right| \le 36\beta n+20 \le 40\beta n$. So it remains only to show that $\mathcal{P}$ has the claimed absorbing property.
 
 Consider any set~$L \subseteq V(H) \setminus V(\mathcal{P})$ with~$|L| \le \beta^2 n$. By possibly adding vertices from~$\{a_5, b_5, c_5\}$ to~$L$, form a set $L'$ with $L \subseteq L' \subseteq L \cup \{a_5, b_5, c_5\}$ such that $|L'|$ is divisible by four. Set $f(\hat{P}) = \hat{P} \setminus L'$, and observe that $f(\hat{P})$ is then a squared-tight-path with the same ends as~$\hat{P}$. Next, arbitrarily partition $L'$ into vertex-disjoint ordered $4$-tuples of distinct vertices $\mathbf{v}_1, \dots, \mathbf{v}_{t} \in V(H)^4$, and write $\mathcal{L} = \{\mathbf{v}_1, \dots, \mathbf{v}_t\}$. The number of these 4-tuples is $|\mathcal{L}| = |L'|/4 \le (\beta^2n+3)/4 \leq \beta^2 n$, and each has at least $\beta^2 n$ absorbers in~$\F$, so we may greedily choose for each~$\mathbf{v} \in \mathcal{L}$ a unique $A_\mathbf{v} \in \F$ which is an absorber for~$\mathbf{v}$. 

Let $\F_\mathcal{L} := \{A_\mathbf{v} \colon \mathbf{v} \in \mathcal{L}\}$ and let $\F_\mathcal{L}^c := \F \setminus \F_\mathcal{L}$. For each~$A \in \F_\mathcal{L}$ set $f(P) = f_{A, \mathbf{v}}(P)$ for each~$P \in \mathcal{P}_A$, and for each~$A \in \F_\mathcal{L}^c$ instead set $f(P) = P$ for each~$P \in \mathcal{P}_A$. So for each~$P \in \mathcal{P}$ the image $f(P)$ is then a squared-tight-path with the same ends as~$P$. Set $\mathcal{Q} := \{f(P) \colon P \in \mathcal{P}\}$, and observe that $V(\mathcal{Q}) = V(\mathcal{P}) \cup L$, as desired.
\end{proof}
 
Finally, we prove Lemma~\ref{lma: absorption lemma}  by connecting the vertex-disjoint squared-tight-paths obtained from Lemma~\ref{lma: disconnected absorption lemma} into a single squared-tight-path. 

\begin{proof}[Proof of Lemma~\ref{lma: absorption lemma}]
Let $\mathcal{P}$ be a collection of vertex-disjoint squared-tight-paths with the properties stated in Lemma~\ref{lma: disconnected absorption lemma}, and write $\mathcal{P} = \{P_1, \dots, P_N\}$. So in particular $N \leq 5\beta n+1$ and $|V(\mathcal{P})| \leq 40 \beta n$.
    For each~$i \in [N]$, let $\mathbf{y}_i$ be the initial triple of~$P_i$ and let $\mathbf{x}_i$ be the final triple of~$P_i$. Apply Lemma~\ref{lma: simpler connecting lemma} with $40\beta, N-1, \mathbf{y}_{i+1}$ for each~$i \in [N-1]$ playing the roles of~$\psi, s, \mathbf{y}_i$ for each~$i\in [s]$ respectively. This gives a collection $\mathcal{P}' = \{P_1', \dots, P_{N-1}'\}$ of vertex-disjoint squared-tight-paths in~$H \setminus X$ such that for each~$i \in [N-1]$ the path $P_i'$ has initial triple $\mathbf{x}_i$ and final triple $\mathbf{y}_{i+1}$, and also so that $|V(\mathcal{P}')| \le \sqrt{40\beta}n$. Set $P = P_1 P_1'P_2P_2'\dots P_{N-1}P_{N-1}'P_N$, so $|V(P)| \le 40\beta n + \sqrt{40\beta}n \le 7\sqrt{\beta} n$.

    Now consider any set $L$ of at most $\beta^2 n$ vertices of~$V(H) \setminus V(P)$. By the absorbing property of~$\mathcal{P}$ there exists a collection $\mathcal{Q}$ of pairwise vertex-disjoint squared-tight-paths in~$H$ with $V(\mathcal{Q}) = V(\mathcal{P}) \cup L$ and a bijection $f: \mathcal{P} \to \mathcal{Q}$ such that for each~$P \in \mathcal{P}$ the paths $P$ and $f(P)$ have the same ends. For each~$i \in [N]$ set $Q_i = f(P_i)$, and set $P' = Q_1P_1'Q_2P_2'\dots Q_{N-1}P_{N-1}' Q_N$. Then $P'$ is a squared-tight-path in~$H$ with $V(P') = V(P) \cup L$ and such that $P$ and $P'$ have the same ends.\end{proof}

\section{Path Cover Lemma} \label{sec: Path Cover Lemma}

In this section, we prove the Path Cover Lemma (Lemma~\ref{lma: path cover lemma}). For this we make use of the following weakened version of a theorem by Keevash and Mycroft~\cite{KeevashMycroft}. For constants $1/n \ll 1/\ell \ll \gamma  \ll \beta \ll 1/k$, the full statement of the result assumes instead the weaker edge-containment condition that for each~$i \in [k-1]$, each~$e \in \J^{(i)}$ is contained in at least $(\frac{k-i}{k} - \gamma)n$ edges of~$\J^{(i+1)}$. The conclusion is then that either $\J^{(k)}$ has a matching covering all but at most $\ell$ vertices, or that there is a subset~$S$ of~$V(\J)$ with $|S| = \lfloor jn/k \rfloor$ for some $j \in [k-1]$ for which at most $\beta n^k$ edges of~$\J^{(k)}$ contain more than $j$ vertices of~$S$. The latter outcome is referred to as a `space barrier', but a straightforward counting argument shows this cannot occur under the stronger edge-containment assumption of the version of the theorem stated below.

\begin{theorem}[{\cite[Theorem 2.4]{KeevashMycroft}}] \label{thm: almost tiling Keevash-Mycroft}
    Let $1/n \ll 1/\ell \ll \alpha \ll 1/k$ and let $\J$ be a $k$-complex on~$n$ vertices with $\J^{(1)} = \{\{v\} \colon v \in V(\J)\}$. If for each~$i \in [k-1]$ each edge $e \in \J^{(i)}$ is contained in at least $(\frac{k-i}{k} + \alpha)n$ edges of~$\J^{(i+1)}$, then $\J^{(k)}$ contains a matching that covers all but at most $\ell$ vertices of~$\J$.  
\end{theorem}

We will apply Theorem~\ref{thm: almost tiling Keevash-Mycroft} with $k=4$ in the proof of the following lemma, showing that every $3$-graph $H$ on $n$ vertices with minimum codegree at least $3n/4+o(n)$ admits an almost-spanning collection of pairwise vertex-disjoint squared-tight-paths.

\begin{lemma}[Path Tiling Lemma] \label{lma: path tiling lemma}
Let $1/n \ll 1/L \ll  \gamma, \alpha \le 1$, and let $H$ be a $3$-graph on $n$ vertices with $\delta_2(H) \ge (3/4 + \alpha)n$. Then there exists a collection of at most $L$ pairwise vertex-disjoint squared-tight-paths in~$H$ which collectively cover all but at most~$\gamma n$~vertices~of~$H$.
\end{lemma}
\begin{proof}
Introduce new constants with $$1/n \ll1/L \ll 1/t_1 \ll 1/t_0 \ll 1/T, 1/r, \eps \ll c \ll \eps_3, d_2 \ll d_3 \ll \beta \ll \theta\ll \gamma, \alpha.$$
    Arbitrarily delete at most $t_1!$ vertices from~$H$ to obtain a subgraph $H'$ on $n'$ vertices where~$t_1!$ divides~$n'$. We then have $\delta_2(H') \ge (3/4 + \alpha)n - t_1! \ge (3/4 + 5\alpha/6)n'$. Apply Lemma~\ref{lem: regular slice}, with $H', n'$ and~$3$ playing the roles of~$G, n$ and $k$ respectively, to obtain a $(\cdot, \cdot, \eps, \eps_3, r)$-regular slice $\J$ for~$H'$ with clusters $V_1, \dots, V_t$, where $t_0 \le t \le t_1$, and with density parameter~$d_2$. Let $R := R_{d_3}(H')$; recall that $V(R) = [t]$ and that vertex $i \in [t]$ corresponds to the cluster $V_i$ of~$\J$. Let $m$ be the number of vertices in each cluster, so $mt = n'$.  Proposition~\ref{Prop: almost codegree in the reduced graph}, with $H', 5\alpha/6, 3$ and $\theta$ playing the roles of~$H, \alpha, \gamma$ and $\beta$ respectively, implies that~$R$ is $(3/4 + 5\alpha/12, \theta)$-dense. By Lemma~\ref{lma: find a strongly dense subgraph}, with $R, t$ and $3/4 + 5\alpha/12$ playing the roles of~$H, n$ and $\mu$ respectively, there is a subgraph $R'$ of~$R$ that is strongly $(3/4 + \alpha/3, 2\theta^{1/4})$-dense, and which also satisfies$|E(R-R')| \le 8\sqrt{\theta}\binom{t}{3}$. 

Let $U = \{i \in V(R') \colon  \deg_{\partial R'}(i) < (1-\theta^{1/8})t\}$. Since $|\partial R'| \ge (1-2\theta^{1/4})\binom{t}{2}$, we have $|U| \le 2\theta^{1/8}t$. Let $R^* = R' \setminus U$ and $t' = |V(R^*)| \ge (1-2\theta^{1/8})t$. For any pair of vertices $uv \in \binom{V(R^*)}{2}$ such that $\deg_{R^*}(uv) > 0$, we have 
\begin{align*}
    \deg_{R^*}(uv) \ge \deg_{R'}(uv) - |U| \ge (3/4 + \alpha/3)t - 2\theta^{1/8}t \ge (3/4 + \alpha/4)t'.
\end{align*}
Note that \begin{align*}
    \left|\binom{V(R^*)}{2}\setminus \partial R^* \right| \le \left|\binom{V(R')}{2} \setminus \partial R' \right| \le 2\theta^{1/4}\binom{t}{2} \le 4\theta^{1/4}\binom{t'}{2}.
\end{align*}
By the above, we deduce that $R^*$ is strongly $(3/4 + \alpha/4, 4\theta^{1/4})$-dense and we have \begin{align} \label{eqn: vtx degree in R^*}
    \delta(\partial R^*) \ge (1-4\theta^{1/4})t' \ge (3/4 + \eps)t'.
\end{align}
Define a $4$-complex $\C$ with $V(\C) = V(R^*)$ in which 
\begin{align*}
    \C^{(1)} &= \{ \{i\} \colon i \in V(R^*) \}, &
    \C^{(2)} &= \partial R^*, &
    \C^{(3)} &= E(R^*), &
    \C^{(4)} &= E(\T(R^*)).
\end{align*}

\begin{claim}
For each~$i \in [3]$ each edge $e \in \C^{(i)}$ is contained in at least $(\frac{4-i}{4} + \eps)t'$ edges of~$\C^{(i+1)}$.
\end{claim}

\begin{proofclaim}
for~$i=1$ the claimed property holds by~\eqref{eqn: vtx degree in R^*}. The fact that $R^*$ is strongly dense implies that each~$ij \in \partial R^*$ is contained in at least $(3/4 + \alpha/4)t'$ edges of~$R^*$. This gives the claimed property for~$i=2$, and also implies that for each edge $ijk \in \C^{(3)}$, each pair of vertices from~$ijk$ has at most $t' - (3/4 + \alpha/4)t' \le (1/4 - \alpha/4)t'$ non-neighbours in~$V(\C)$. It follows that the number of options to choose a vertex $\ell$ such that $ijk\ell \in E(\T(R^*))$ is at least $t' - 3(1/4 - \alpha/4)t' = (1/4 + 3\alpha/4)t' \ge (1/4 + \eps)t'$, giving the claimed property for~$i=3$ also. 
\end{proofclaim}

By Theorem~\ref{thm: almost tiling Keevash-Mycroft}, with $\C, \eps$ and $T$ playing the roles of~$\J, \alpha$ and $\ell$ respectively, the $4$-graph~$\C^{(4)}$ contains a matching covering all but at most $T$ vertices of~$\C$. This is a $K_4^3$-tiling $\mathscr{K}$ in~$R$ covering all but at most $2\theta^{1/8}t + T \le 3\theta^{1/8}t$ vertices of~$R$.   

\begin{claim} \label{clm: path tiling in each tetrahedron}
    Let $i_1, i_2, i_3, i_4$ be the vertices of a copy of~$K_4^3$ in~$R$. Then there is a collection $\mathcal{P}$ of at most $1/c$ pairwise vertex-disjoint squared-tight-paths in~$H[\bigcup_{j \in [4]} V_{i_j}]$ which covers all but at most $\beta m$ vertices of each of~$V_{i_1}, V_{i_2}, V_{i_3}$ and $V_{i_4}$.
\end{claim}

\begin{proofclaim}
 Let~$\cG^{(3)}$ be a squared-tight-path~$v_1v_2...v_{4cm-1} v_{4cm}$ where~$v_k \in V_{i_j}$ for all $k \in [4cm]$ such that~$k \equiv j \pmod{4}$. Let~$\cG$ be the down-closure of~$\cG^{(3)}$, so~$\cG$ is a~$4$-partite~$3$-complex with vertex classes~$X_1, \dots, X_4$ where~$X_{j} \subseteq V_{i_j} \text{ and } |X_j| = cm$ for all $j \in [4]$. Note that any vertex in~$\cG$ is in at most~$9$ edges in~$\cG^{(3)}$, therefore~$\Delta = \Delta(\cG) = 9 + 6 + 1 = 16$. Let~$\Hy$ be the~$4$-partite~$3$-complex obtained from~$\J[\bigcup_{j \in [4]} V_{i_j}]$ by adding the edges of~$H'$ supported on~$\J[\bigcup_{( j, k, \ell) \in \binom{[4]}{3}} V_{i_j}V_{i_k}V_{i_{\ell}}]$. Therefore~$\Hy$ is a~$(\mathbf{d}, \eps_k, \eps, r)$-regular~$4$-partite~$3$-complex with vertex classes~$V_{i_1}, \dots, V_{i_4}$ all of size~$m$, respecting the partition of~$\cG$.

 Let $\mathcal{P}$ be a maximal collection of vertex-disjoint squared-tight-paths in~$H[\bigcup_{j \in [4]} V_{i_j}]$, each of which uses exactly $cm$ vertices from each of~$V_{i_1}, \dots, V_{i_4}$. In particular, the latter implies that $|\mathcal{P}| < 1/c$. Let $V_{i_j}' := V_{i_j} \setminus V(\mathcal{P})$ for each~$j \in [4]$. If $|\mathcal{P}| < (1-\beta)/c$, then we have $|V_{i_j}'| \ge \beta m$ for each~$j \in [4]$.  Lemma~\ref{lma: regular restriction}, with~$\Hy, 3, 4, \beta, \{V_{i_j}\}_{j \in [4]}$ playing the roles of~$\cG, k, s, \alpha, \{V_i\}_{i \in [s]}$ respectively, then implies that the induced subcomplex~$\Hy[V_{i_1}', \dots, V_{i_4}']$ is~$(\mathbf{d}, \sqrt{\eps_3}, \sqrt{\eps}, r)$-regular. So we may apply Lemma~\ref{lma: Embedding Lemma}, with $\Hy[V_{i_1}', \dots, V_{i_4}']$, $3$, $c/\beta$, $\beta m$, $\sqrt{\eps_3}$ and $\sqrt{\eps}$ playing the role of $\Hy, k, c, m, \eps_k$ and $\eps$ respectively, to obtain a partition-respecting copy of~$\cG$ in~$\Hy[V_{i_1}', \dots, V_{i_4}']$. This gives us a squared-tight-path $P$ in~$H[\bigcup_{j \in [4]} V_{i_j}]$ which uses exactly $cm$ vertices from each of~$V_{i_1}, \dots, V_{i_4}$, so the collection $\mathcal{P} \cup P$ contradicts the maximality of~$\mathcal{P}$. Therefore, we must have $(1-\beta)/c \le |\mathcal{P}|$. It follows that for each~$j \in [4]$ we have $|V_{i_j} \setminus V(\mathcal{P})| \le m - |\mathcal{P}|cm \le \beta m$. This proves the claim.
\end{proofclaim}

Applying Claim~\ref{clm: path tiling in each tetrahedron} to each member of~$\mathscr{K}$ we obtain a collection of at most $|\mathscr{K}|/c \le t_1/4c \le L$ vertex-disjoint squared-tight-paths in~$H'$ which collectively cover all but at most $\beta mt + 3\theta^{1/8}mt + t_1! \le \gamma n$ vertices of~$H$. \end{proof}

We are now ready to prove Lemma~\ref{lma: path cover lemma}. We do this by connecting the squared-tight-paths obtained from Lemma~\ref{lma: path tiling lemma} and the specified ends $\mathbf{e}_1$ and $\mathbf{e}_2$ into a single path.

\begin{proof}[Proof of Lemma~\ref{lma: path cover lemma}]
Let~$ \mathbf{e}_1$ and~$\mathbf{e}_2$ be disjoint ordered triples which are edges of~$H$. Let~$H' := H \setminus (V(\mathbf{e}_1)\cup V(\mathbf{e}_2))$ and~$n' := |V(H')|$. Note that~$\delta_2(H') \ge (7/9+\alpha)n-6 \ge (7/9+3\alpha/4)n'$. 

Choose a set~$W \subseteq V(H')$ of size~$|W| = \gamma n'/2$ uniformly at random. Observe that for each pair $u, v$ of distinct vertices of~$H$, both $\deg_{H} (uv, W)$ and~$\deg_{H}(uv, V(H') \setminus W)$ are hypergeometric random variables, with expectations~$\mathbb{E}\deg_{H} (uv, W) = (7/9+3\alpha/4)\gamma n'/2$ and~$\mathbb{E}\deg_{H} (uv, V(H') \setminus W) = (7/9 + 3\alpha/4)(1-\gamma/2)n'$. So by applying Hoeffding's inequality (Lemma~\ref{Lma: Chernoff hyp}) and taking a union bound over all pairs of vertices, we find that with high probability every pair $u,v$ of distinct vertices of~$H$ satisfies  \begin{align} \label{eqn: codegree in and out of W}
    \deg_{H} (uv, W) \ge \left(\frac{7}{9} + \frac{\alpha}{2}\right)|W| \text{ and } \deg_{H} \left(uv, V(H') \setminus W\right) \ge \left(\frac{7}{9} + \frac{\alpha}{2}\right)|V(H')\setminus W|.
\end{align} 
Fix~$W$ for which this event holds. Let~$n'' := |V(H') \setminus W|$, so $\delta_2(H'\setminus W) \ge (7/9 + \alpha/2)n''$. 

By Lemma~\ref{lma: path tiling lemma}, with $H' \setminus W, n'', \alpha/2$ and $\gamma/2$ playing the roles of $H, n, \alpha$ and $\gamma$ respectively, there exists a collection of at most~$L$ squared-tight-paths $P_1, \dots, P_{\ell}$ in~$H' \setminus W$ which collectively cover all but at most $\gamma n''/2$ vertices of~$H' \setminus W$.  We now connect these squared-tight-paths using~$W$. For each~$i \in [\ell]$ let $\mathbf{p}_i$ be the initial triple of~$P_i$ and let $\mathbf{q}_i$ be the final triple of~$P_i$. Also let $P_0 := \mathbf{p}_0 = \mathbf{q}_0 = \mathbf{e}_1$ and $P_{\ell+1} := \mathbf{p}_{\ell+1} = \mathbf{q}_{\ell+1} = \mathbf{e}_2$. 
\begin{claim}
   For each~$0 \le i \le \ell$, there exist disjoint ordered triples $\mathbf{x}_i$ and $\mathbf{y}_i$ whose vertices are in~$W$ such that \begin{enumerate}[noitemsep, label={\rm(\roman*)}]
    \item $\mathbf{x}_i$ and $\mathbf{y}_i$ are edges of~$H[W]$,
    \item $\mathbf{x}_i\mathbf{p}_i$ and $\mathbf{q}_i\mathbf{y}_i$ are both squared-tight-paths, and
    \item the triples $\mathbf{x}_i$ and $\mathbf{y}_i$ for each~$0 \leq i \leq \ell$ are all pairwise vertex-disjoint.
\end{enumerate}

\end{claim}
\begin{proofclaim}
We write $\mathbf{p}_i = p_ip_i'p_i''$ and $\mathbf{q_i}$ similarly. For each~$0 \le i \le \ell$ we will pick vertices $x_i, x_i',x_i''$ in order and write $\mathbf{x}_i = x_ix_i'x_i''$ and similarly for~$\mathbf{y}_i$. To see that it is possible to choose such triples, for each~$0 \leq i\ \leq \ell$ in turn first choose $x_i''$ in~$N_H(p_ip_i', W) \cap N_H(p_ip_i'', W) \cap N_H(p_i'p_i'', W)$, then choose $x_i'$ in~$N_H(p_ip_i', W) \cap N_H(p_ix_i'', W)\cap N_H(p_i'x_i'', W)$, then choose $x_i$ in~$N_H(p_ix_i', W) \cap N_H(p_ix_i'', W) \cap N_H(x_i'x_i'', W)$, and similarly choose $y_i, y_i', y_i''$ in that order. For each choice,~\eqref{eqn: codegree in and out of W} ensures that the relevant intersection of neighbourhoods has size at least $3(7/9+\alpha/2) |W| - 2|W| \geq |W|/3 \geq 6(\ell+1)$, so there is always an available vertex to choose which has not previously been used.\end{proofclaim}

Having picked $\mathbf{x}_i$ and $\mathbf{y}_i$ for each~$0 \le i \le \ell$, we now apply Lemma~\ref{lma: simpler connecting lemma}, with $W, 1/L, \theta n', \ell+1, \mathbf{y}_i$ and $\mathbf{x}_{i+1}$ for each~$0\le i\le\ell$ playing the roles of $H, \psi, n, s, \mathbf{x}_i$ and $\mathbf{y}_i$ for each~$i \in [s]$ respectively, to obtain a set of vertex-disjoint squared-tight-paths $P_0', \dots, P_{\ell}'$ such that for each~$0\le i \le \ell$, the path $P_i'$ has initial triple $\mathbf{y}_i$ and final triple $\mathbf{x}_{i+1}$. It follows that ${P} = P_0 P_0' P_1 P_1' \dots P_{\ell}'P_{\ell+1}$ is a squared-tight-path from~$P_0 = \mathbf{e}_1$ to~$P_{\ell+1} = \mathbf{e}_2$ covering all but at most $\gamma n''/2 + |W| \le \gamma n$ vertices of~$H$. 
\end{proof}

\printbibliography   
\end{document}